\documentclass[12pt]{amsart}

\usepackage{euler, amsfonts, amssymb, amsmath, amsthm, amsopn}
\usepackage{latexsym, epsfig, epic, xcolor, upgreek}
\usepackage{hyperref} 
\usepackage{wasysym} 
\usepackage{ stmaryrd }

\usepackage{accents}

\renewcommand\eqref[1]{(\ref{#1})}

\newcommand\lie[1]{{\mathfrak{#1}}}

\newcommand\iso{{\ \cong\ }}
\newcommand\tensor{{\otimes}}

\newcommand\calO{{\mathcal O}}

\newtheorem{Theorem}{Theorem} 
\newtheorem{Proposition}{Proposition} 
\newtheorem{Lemma}{Lemma}

\newtheorem{Corollary}{Corollary}
\newtheorem*{Corollary*}{Corollary}
 
\newtheorem*{Theorem*}{Theorem}
\theoremstyle{remark}
\newtheorem{Example}{Example}

\newcommand\onto{\mathop{\twoheadrightarrow}}
\newcommand\into{\operatorname*{\hookrightarrow}}

\newcommand\calL{{\mathcal L}}

\newcommand\CC{{\mathbb C}}

\newcommand\junk[1]{}

\theoremstyle{plain}

\newcommand\wt[1]{{\widetilde #1}}
\newcommand{\todo}[2][0.9]{\vspace{1 mm}\par \noindent
 \framebox{{\color{red}\begin{minipage}[c]{#1
\textwidth} \tt #2 \end{minipage}}}\vspace{1 mm}\par}

\newcommand\AKrem[1]{{\color{teal} #1}}
\newcommand\acirc{{\mathring a}}
\newcommand\dcirc{{\mathring d}}
\newcommand\deltacirc{{\accentset{\circ} \delta}}
\newcommand\deltop{{\overset{\tiny\Leftcircle} \delta}}

\newcommand\Demopisobaric{\delta}
\newcommand\Demop{\deltacirc}

\newcommand\ea{e^\alpha}
\newcommand\ema{e^{-\alpha}}
\newcommand\da{\Demopisobaric_\alpha}
\newcommand\bda{\Demop_\alpha}
\newcommand\ra{r_\alpha}

\numberwithin{equation}{section}

\begin{document}

\author{Rebecca Goldin}
\address{ George Mason University, Fairfax, VA 22030, USA {\tt rgoldin@gmu.edu}}
\author{Allen Knutson}
\address{Cornell University, Ithaca, NY 14853, USA {\tt allenk@math.cornell.edu}}

\title{Schubert structure operators and $K_T^*(G/B)$}
\date{\today}

\dedicatory{In loving memory of our friend Bert Kostant}
\maketitle

\begin{abstract}
  We prove a formula for the structure constants of multiplication of
  equivariant Schubert classes in both equivariant cohomology
  and equivariant $K$-theory of Kac-Moody flag varieties $G/B$. 
  We introduce new operators whose coefficients compute
  these 
 (in a manifestly polynomial, but not
  positive, way), resulting in a formula much like and generalizing
  the positive Andersen-Jantzen-Soergel/Billey and Graham/Willems
  formul\ae\ for the restriction of classes to fixed points. 
  Our proof involves Bott-Samelson manifolds, and in particular, 
  the ($K$-)cohomology basis dual to the ($K$-)homology
  basis 
  consisting of classes of sub-Bott-Samelson manifolds.
\end{abstract}

{ \small\tableofcontents}
%
%

\newcommand\ZZ{{\mathbb Z}}

\newcommand\defn[1]{{\bf #1}}

\newcommand\com[1]{{\color{red} #1}}

\section{Introduction and the main theorems}

Fix a complex reductive (or even Kac-Moody) Lie group $G$ and 
maximal torus $T\leq G$, for example 
$G=GL_n(\CC)$ and $T$ the diagonal matrices.  
Fix opposed Borel subgroups $B,B_-$ with intersection $T$. This choice
results in a length function $\ell$ on $W = N(T)/T$ and a set
$\{\alpha_i\}$ of simple roots.  The quotient $G/B$ is the associated
\defn{flag variety}\footnote{Following \cite{KostantKumarH}, we are
  using the ``thin'' flag variety, namely the ind-scheme made as an
  inductive union of finite-dimensional $B$-orbit closures.}
and the left $T$-action on $G/B$ has isolated
fixed points $\{wB/B\ :\ w\in W\}$, where $W:= N(T)/T$ is the \defn{Weyl group}.

\junk{
{\color{blue} We probably don't need this paragraph, i.e. can assume the readers know a little more about Lie groups.} In the case that $G = GL_n(\CC)$ and $B =$ upper-triangular matrices,
$G/B$ is (uniquely) $G$-isomorphic to the set of complete flag
manifolds $Fl(\CC^n)$. The fixed points $N(T)B/B$ of the $T$-action
correspond, under that isomorphism, to coordinate flags in
$Fl(\CC^n)$. In particular, there are $n!$ such flags, corresponding
to elements in the Weyl group $W \cong S_n,$ the symmetric group on
$n$ letters. 
}

We denote by $H^*_T$ the $T$-equivariant cohomology of a point with
coefficients in $\ZZ$, and recall that $H^*_T$ is the polynomial ring
$Sym(T^*)$ over $\ZZ$ in the weight lattice $T^* := Hom(T,\CC^\times)$.
The equivariant cohomology $H_T^*(G/B)$ 
is a free $H^*_T$-module with a basis given by Schubert classes 
(recalled below). Our references for equivariant (co)homology of $G/B$ are
\cite{BrionH,KostantKumarH,KumarNori}. 
Similarly we let $K_T^*$ denote the $T$-equivariant $K$-theory of a point,
a Laurent polynomial ring of characters of $T$-reps
\cite{KostantKumarK}. To best compare/distinguish the two, we write a
typical element of $K_T^*$ as a finite sum $\sum_{\lambda\in T^*} m_\lambda e^\lambda$.

\junk{
{\color{blue} We denote by $K_T^*$ the equivariant $K$-theory... some words about which we are using... intrdocutory remarks.}
}

\subsection{The usual operators, 
  following \cite{KostantKumarH,KostantKumarK}}
\label{ssec:usual}

Let $\ZZ[\partial]$
denote the \defn{nil Hecke algebra} 
with $\ZZ$-basis $\{\partial_w\ :\ w\in W\}$, whose products are
defined by 
$$
\partial_w \partial_{r_\alpha} 
:= \begin{cases}\partial_{w r_\alpha} &\mbox{if $\ell(wr_\alpha)=\ell(w)+1$}\\
0 & \mbox{otherwise, i.e. if $\ell(wr_\alpha) = \ell(w)-1$}
\end{cases}
$$
for simple reflections $r_\alpha$.
These $\{\partial_w\}$ act on the polynomial ring $H^*_T$ 
as follows: for each simple root $\alpha$ with 
reflection $r_\alpha$, the \defn{divided difference operator}
$\partial_{r_\alpha} := \partial_\alpha$ is defined by
$$ \partial_\alpha \cdot f := \frac{f-r_\alpha f}{\alpha}. $$  
The nil Hecke algebra acts on the first factor in the tensor product
$H^*_T \tensor_{\ZZ} H^*_T$, and this action descends to the
quotient $H^*_T \ \tensor_{(H^*_T)^W}\ H^*_T$.
This latter ring has a well-defined map 
$\lambda \tensor \mu \mapsto \lambda c_1(\calL_{\mu})
\in H_T^*(G/B)$
called the \defn{equivariant Borel presentation} of $H^*_T(G/B)$,
which is a rational (and for $G=GL_n$, an integral) isomorphism.
(Here $\calL_\mu$ is the Borel-Weil line bundle $G \times^B \CC_\mu$,
where $\CC_\mu$ is the $1$-dimensional representation of $B$, neither of which
will we be using again.)

Similarly, we define the \defn{zero Hecke algebra} $\ZZ[\deltop]$
with $\ZZ$-basis $\{\deltop_w\ :\ w\in W\}$, whose \defn{Demazure products} 
are defined by
$$
\deltop_w \deltop_{r_\alpha} 
:= \begin{cases}\deltop_{w r_\alpha} &\mbox{if $\ell(wr_\alpha)=\ell(w)+1$}\\
\deltop_w & \mbox{otherwise, i.e. if $\ell(wr_\alpha) = \ell(w)-1$}
\end{cases}
$$
and which has {\em two} actions on $K_T^*$, 
by two flavors of \defn{Demazure operators}
$$ \Demop_\alpha \cdot f := \frac{f - r_\alpha\cdot f}{1-e^{-\alpha}}
\qquad\qquad
\Demopisobaric_\alpha \cdot f := \frac{f - e^{-\alpha}r_\alpha\cdot f}{1-e^{-\alpha}} $$
As algebras we have $\ZZ[\Demop] \iso \ZZ[\Demopisobaric] \iso \ZZ[\deltop]$,
but occasionally we will use the first two when anticipating their 
action somewhere.



\subsection{The bases of ($K$-)(co)homology}

Define $X^v := \overline{BvB}/B \subset G/B$ 
with equivariant homology class $[X^v] \in H^T_*(G/B)$ (our reference for
equivariant homology being \cite{BrionH}).\footnote{For the entirety of this paper, we use $\subset$ to denote inclusion, and allow for equality.} As these $\{[X^v]\}$ form an $H^*_T$-basis and
$G/B$ enjoys equivariant Poincar\'e duality,
we can define the dual basis $\{S_w \in H^*_T(G/B)\}$ 
of \defn{Schubert classes} by $\langle S_w, [X^v] \rangle = \delta_{wv}$.
Here $\langle,\rangle$ denotes the Alexander pairing, of (equivariant)
cap-product followed by pushforward to a point.
In fact $S_w$ is the Poincar\'e dual to the
(finite-codimensional) subvariety $\overline{B_- w B}/B$.
We don't strictly need to decide which of $\{X^v\},\{X_v\}$ need be called
Schubert vs. opposite Schubert varieties, but of necessity the
finite-dimensional varieties $X^v$ define homology classes and
the finite-codimensional varieties $X_v$ define cohomology classes.

In $K$-homology
\junk{Following at least some of the literature,
  we use $K()$ without a $*$ to denote $K$-cohomology, but for clarity
  we will need it in $K_*$ to denote $K$-homology.}
$K_*(G/B)$, unlike $H_*(G/B)$, there are {\em two}\footnote{%
  There is a wholly separate issue in the finite-dimensional case that
  $H_*^T(G/B)$ has two natural bases, $\{[X^w]\}$ and
  $\{w_0\circ[X^w]\}$, that coincide once one passes to nonequivariant
  homology. With that in mind, $K^T_*(G/B)$ has {\em four} natural
  bases that nonequivariantly become the two we're discussing here.}
natural bases: one is the basis of structure sheaves $\{\calO_{X^w}\}$
coming from functions on $X^w$, the other being the basis of ideal
sheaves $\{I_{X^w}\}$ coming from functions on $X^w$ that vanish on
its ``boundary'' $\cup_{w' < w} X^{w'}$. Each basis has an
evident equivariant extension to a $K_T^*$-basis of $K^T_*(G/B)$, and
the change-of-basis matrices are well known (and very simple even
equivariantly):
$$ [\calO_{X^w}] = \sum_{w'\leq w} [I_{X^{w'}}], \qquad\qquad
[I_{X^w}] = \sum_{w'\leq w} (-1)^{\ell(w)-\ell(w')} [\calO_{X^{w'}}] $$
These two bases of $K^T_*(G/B)$ are indistinguishable at the coarser
level of homology, owing to the short exact sequence
$0 \to I_{X^w} \into \calO_{X^w} \onto \calO_{\cup_{w' < w} X^{w'}} \to 0$
whose cokernel has lower-dimensional support.

To the bases $\{[X^w]\}$, $\{[\calO_{X^w}]\}$, $\{[I_{X^w}]\}$ 
in ($K$-)homology,
we respectively associate the dual bases (under the Alexander pairings) 
$\{S_w\}$, $\{\xi_w^\circ\}$, and $\{\xi_w\}$ in ($K$-)cohomology.
As in $H^*_T(G/B)$ the classes $\xi^\circ_w,\xi_w \in K_T^*(G/B)$ 
have geometric interpretations, being the Poincar\'e dual classes associated 
to the $K^T_*$-homology classes $[I_{X_w}]$, $[\calO_{X_w}]$ respectively.

\junk{
  Let $S_v \in H_T^*(G/B)$ denote the Poincar\'e dual to 
  $X_v := \overline{B_- v B}/B$, and similarly let $S^v\in H_T^*(G/B)$
  denote the dual to $X^v := \overline{B v B}/B$. 
In particular, $S^{w_0}= S_1 =1$, where $w_0$
is the longest word of $W$. We call $\{S^v\}_{v\in W}$ the 
\defn{Schubert classes} and $\{S_v\}_{v\in W}$ the 
\defn{opposite Schubert classes}. 
}

The nil Hecke algebra $\ZZ[\partial]$ acts simply in 
the basis $\{S_v\}_{v\in W}$: in particular,
in the case $G$ finite-dimensional we have
$\partial_w \cdot S_{w_0} = S_{w w_0 }$ for each $w\in W$ 
(since we act on the left factor in the Borel presentation),
with similar statement
$\Demopisobaric_w\cdot \xi_{w_0} = \xi_{w w_0 }$ in $K$-theory.
The geometric meaning of the $\Demop$ operators is much less clear.
%
\junk{The bases $\{\xi_w\}$, $\{w_0\cdot \xi^\circ_w\}$ are dual bases
  under the $K_T^*$-valued pairing
  $K_T^*(G/B)^{\tensor 2} \xrightarrow{\cup} K_T^*(G/B)^{\tensor 2} 
  \xrightarrow{\int} K_T^*$. Consequently, the operators $\Demop_w$,
  $w_0 \Demopisobaric_w^\circ w_0$ are adjoint.}



\subsection{Multiplying in these bases of ($K$-)cohomology}
\label{ssec:mult}

The structure constants $c_{uv}^w\in H_T^*$ of multiplication
are defined by the relation in $H^*_T(G/B)$
\begin{equation}\label{eq:structureconstants}
S_u S_v = \sum_w c_{uv}^w S_w
\end{equation}
These polynomials $c_{uv}^w$ are known to be positive
in the following sense 
\cite{Graham:Positivity}: 
when written (uniquely) as a sum of
monomials in the simple roots $\{\alpha_i\}$, 
each monomial has a non-negative coefficient.

Similarly, in $K_T^*(G/B)$, the products
\begin{equation}\label{eq:Kcoefficientsdefinition}
 \xi_w \xi_v = \sum_u a_{wv}^u\, \xi_u
\qquad\qquad
 \xi^\circ_w \xi^\circ_v = \sum_u \acirc_{wv}^u\, \xi^\circ_u
\end{equation}
define classes $a_{wv}^u, \acirc_{wv}^u \in K_T^*$,
with subtler positivity properties proven in \cite{AGM}.\footnote{In \cite{GrahamKumar}, the authors denote the coefficients $\acirc_{wv}^u$ by $p_{w,v}^u$, and the classes $\xi^\circ_u$ by $\xi^u_B$.}

It is a very famous problem to compute these in a manifestly positive way,
only solved in special cases such as $u,v\in W^P$ where $G/P$ is a
Grassmannian or $2$-step flag manifold \cite{KT,Buch,KZJ17},
or when $W$ is a free 
Coxeter group \cite[Theorem 2.5]{BerensteinRichmond}.
Another solved case is $u=w$, in which case $c_{wv}^w$ is computed
positively by the AJS/Billey formula \cite{AJS,Billey} 
(recalled below)
for the point restrictions $S_w|_v = c_{wv}^v$ of Schubert classes,
$\acirc_{vw}^w$ similarly in \cite{Graham,WillemsH}, and
$a_{vw}^w$ in \cite[Theorem 3.12]{Graham}
(see also \cite{LenartZainoulline}).
In this paper, we prove formul\ae\ for the $\{c_{uv}^w, a_{uv}^w, \acirc_{uv}^w\}$
in terms of certain compositions of operators in the nil/zero Hecke algebras, 
applied to $1$.
Along the way, we reprove the AJS/Billey and Graham/Willems formul\ae;
more specifically, our nonpositive formul\ae\ reduce to those
positive formul\ae\ in the special case $u=w$.

Given a word $R$ in $G$'s simple reflections, let $\prod R$ denote
its ordinary product and $\wt\prod R$ its Demazure product.
Given a true or false statement $\tau$, let $[\tau] = 1$ if true, $0$ if false.

\begin{Theorem}\label{thm:main}
  Let $Q$ be a reduced word with product $w$. Then
  $$ c_{uv}^w = 
  \sum_{R,S\subset Q \text{ reduced}\atop \prod R = u, \ \prod S = v}
  \prod_Q \left(
    {\alpha_q}^{\ [q \in R \cap S]} \
    { \partial_q}^{\ [q \notin R \cup S]} \
    r_q \right) \cdot 1 
  =
  \sum_{R,S\subset Q \text{ reduced}\atop \prod R = u, \ \prod S = v}
  \prod_Q \left(
    {\alpha_q}^{\ [q \in R \cap S]} \
    r_q \ 
    {(- \partial_q)}^{\ [q \notin R \cup S]}
    \right) \cdot 1 $$
  where the exponent 
  ``$[\sigma]$'' is $1$ if the statement $\sigma$ is true, $0$ if false. 
  Similarly, for $Q$ any word whose Demazure product $\wt\prod Q$ is $w$, 
  the $K$-theoretic structure constants are computable by
  \begin{eqnarray*}
    a_{uv}^w &=& (-1)^{\ell(u)+\ell(v)-\ell(w)} \\
    && \sum_{R,S\subset Q\atop \wt\prod R = u, \ \wt\prod S = v}
       (-1)^{|Q|-|R|-|S|}  \left( \prod_{q\in Q} 
       (e^{+\alpha_q})^{[q\notin {R\cup S}]}  (1-e^{-\alpha_q})^{[q\in R\cap S]} \,
       r_q (-\Demop_q)^{[q\notin {R}\cup {S}]}\right) \cdot 1 \\
    \acirc_{uv}^w 
    &=& \sum_{R,S\subset Q\atop \wt\prod R = u, \ \wt\prod S = v}
        \phantom{(-1)^{|Q|-|R|-|S}}
        \left( \prod_{q\in Q} {(e^{-\alpha_q})^{ [q\notin {R\cap S}] }}
        {(1-e^{-\alpha_q})}^{[q\in R \cap S]} \, r_q 
        (-\Demopisobaric_q)^{[q\notin  R \cup S]} \right) \cdot 1. 
  \end{eqnarray*}
\end{Theorem}


We can recover $H_T$ as the associated graded of $K_T^*$ with respect to
the $\langle \{1 - e^\lambda\ :\ \lambda\in T^*\} \rangle$-adic filtration.
(In practice this means replacing each $1 - e^{-\beta}$ by $\beta$,
and any remaining $e^\beta$ factors by $1$, then finally throwing away
any lower-degree terms.
Geometrically, passage to the associated graded ring corresponds to 
degenerating $T \iso {\rm Spec}\ K_T^*$ to 
the normal cone $\lie{t} \iso {\rm Spec}\ H_T$ at the identity element,
i.e. the group to its Lie algebra.)
This visibly takes $a_{uv}^w, \acirc_{uv}^w \mapsto c_{uv}^w$, 
at least to the second formula for $c_{uv}^w$, 
which is why we included that version despite
being uglier than the first. The $H^*_T$ commutation relation
$\partial_\alpha r_\alpha = - r_\alpha \partial_\alpha$ becomes
more complicated in $K_T^*$-theory,
$\Demopisobaric_\alpha  r_\alpha = -e^\alpha r_\alpha \Demopisobaric_\alpha + (1 + e^\alpha)$,
obstructing our discovery of a formula for $a_{uv}^w$ with the $r$s
in the same place as in our first $c_{uv}^w$ formula.

{\em Example.}
Let $Q = 1\,2\,1$ so $w = r_1 r_2 r_1$, $u = r_1$, $v = r_1 r_2$ all in 
$S_3$ the Weyl group of $GL_3$. Then $R \in \{1--,--1\}$, $S = 1\,2\, -$
as subwords of $1\,2\,1$, in our sum
$$
  c_{r_1,\ r_1r_2}^{r_1r_2r_1} 
  = (\alpha_1 r_1\  r_2\  \partial_1 r_1)\cdot 1 + (r_1\  r_2\  r_1) \cdot 1 
  = 0 + 1
$$
whereas if we change $v$ to $r_2r_1$ so $S = -\,2\,1$, then
$$
  c_{r_1,\ r_2r_1}^{r_1r_2r_1} 
  = (r_1\  r_2\  r_1) \cdot 1 + (\partial_1 r_1\  r_2\  \alpha_1 r_1)\cdot 1 
  = 1 + \partial_1\cdot \alpha_2 = 0.
$$

{\em Example.}
Let $Q = 1\,2\,3\,1\,2$, so $w = r_1 r_2 r_3 r_1 r_2 =[3421]$ in one-line
notation, and take $u=r_2r_3r_2=[1432],\ v=r_1r_2r_1=[3214]$. 
Then $R = -\,2\,3-2$ and $S\in \{1\,2-1-,-2-1\,2\}$ so we have
\begin{align*}
c_{uv}^w &= 
(r_1\  \alpha_2 r_2\   r_3\  r_1\  r_2 
+ \partial_1r_1\  \alpha_2 r_2\   r_3\  r_1\ \alpha_2 r_2 )
\cdot 1\\
&= (\alpha_1+\alpha_2)\cdot 1 + \partial_1(\alpha_1+\alpha_2)(\alpha_2+\alpha_3)\cdot 1\\
& =  \alpha_1+\alpha_2 + \partial_1 (\alpha_1+\alpha_2)\alpha_2 \cdot 1 + \partial_1 (\alpha_1+\alpha_2)\alpha_3\cdot 1\\
&=\alpha_1+\alpha_2 +0+ \alpha_3.
\end{align*}
For the $K$-theory formul\ae\ we have three possible subwords
$S \in \{1\,2-1-,-2-1\,2, 1\, 2-1\, 2\}$, so 
\begin{eqnarray*}
  a_{uv}^w &=& \left(
\begin{array}{rccc ccccc cc}
   &&              r_1            & (1-e^{-\alpha_2}) r_2& r_3& r_1& r_2 &\\
   &+& e^{\alpha_1} r_1 (-\Demop_1)& (1-e^{-\alpha_2}) r_2& r_3& r_1& 
                                       (1-e^{-\alpha_2}) r_2 &\\
  &-&                  r_1   &  (1-e^{-\alpha_2}) r_2& r_3& r_1& (1-e^{-\alpha_2}) r_2 &&
\end{array}\right)      \cdot 1 \\
  &=& (1-e^{-\alpha_1-\alpha_2})
      + (e^{-\alpha_2} (1 - e^{-\alpha_3}))
      - (1-e^{-\alpha_1-\alpha_2})    (1-e^{-\alpha_2-\alpha_3})  \\
  &=& e^{-\alpha_2}(1 - e^{-\alpha_1-\alpha_2-\alpha_3})
\end{eqnarray*}
%
%
which $\langle \{1-e^\lambda\ :\ \lambda\in T^*\}\rangle$-adically
associated-grades, as it must, to the $c_{uv}^w$ computed above.

We now recall the AJS/Billey formula. The $T$-invariant inclusion $i$
of $T$-fixed points into $G/B$ results in a map in equivariant cohomology:
\begin{equation}\label{eq:fixedpointinclusion}
  i^*: H^*_T(G/B) 
  \longrightarrow \bigoplus_{w \in W} H_T^*(wB/B) \cong \bigoplus_{w \in W} H_T^*
\end{equation}
and $i^*$ is well known to be an {\em injection.}  The inclusion
$i_w: wB/B\hookrightarrow G/B$ induces the projection to the $w$-term
in this sum, so we may write $i^* = \oplus_{w\in W}\ i^*_w$.


\junk{
  Restricting \eqref{eq:structureconstants} to $u$,
  and using the upper triangularity once again,
  $$
  S_u|_u S_v|_u  = (S_u S_v)|_u = \sum_{w\geq u,v} c_{uv}^w S_w|_u 
  = \sum_{w = u} c_{uv}^w S_w|_u  = c_{uv}^u S_u|_u
  \qquad \therefore c_{uv}^u = S_v|_u. $$
} 

For any $v, w \in W$, the 
\defn{point restriction} $S_v|_w\in H_T^*$  is defined by $i_w^*(S_v)$, 
i.e. the image of $S_v$ under the map $i^*$ in \eqref{eq:fixedpointinclusion}, 
projected to the $w$ summand. 
Since  \eqref{eq:fixedpointinclusion} is an inclusion,  
each Schubert class $S_v$ is described fully by the list
$\{i_w^*(S_v)\ :\ w\in W\}$ of these restrictions.
Note that $S_w|_u \neq 0$ implies $uB/B \in \overline{B_- wB}/B$, i.e. 
$u \geq w$ in \defn{Bruhat order}, and in fact the converse is also true.
This ``upper triangularity of the support'' will be useful just below.

In the case $u=w$, the relation \eqref{eq:structureconstants} and this
upper triangularity imply that $c_{uv}^w = S_v|_w$. After choosing $Q$
a reduced word for $w$, the only choice of reduced word $R$ for $u$ is
$Q$ itself.  The formula in Theorem \ref{thm:main} for
$a_{wv}^w$ thus simplifies to
$$ S_v|_w = 
  \sum_{R\subset Q \text{ reduced}\atop 
    \prod R = v}
  \prod_Q \left(    { \alpha_q}^{[q \in R]} \    r_q \right) \cdot 1, $$
which, as we explain at the beginning of \S \ref{sec:ajsb},
is just a restatement of the AJS/Billey formula \cite{AJS,Billey}.

The corresponding $K$-theoretic special cases are
\begin{eqnarray*}
  a_{wv}^w = \xi_v|_w &=& 
\sum_{R\subset Q\atop \wt\prod R = v} (-1)^{|R|-\ell(v)}
  \left( \prod_{q\in Q} 
  {(1-e^{-\alpha_q})}^{[q\in R]} \, r_q \, \right) \cdot 1
\\
  \acirc_{wv}^w = \xi^\circ_v|_w 
  &=& \sum_{R\subset Q\atop \wt\prod R = v}
    \prod_{q\in Q} \left(
  (e^{-\alpha_q})^{[q\notin { R}]}  (1-e^{-\alpha_q})^{[q\in R]}  r_q\right) \cdot 1
\end{eqnarray*}
where the first matches \cite[Theorem 3.12]{Graham} and the second
matches \cite[Theorem 3.7]{Graham}, \cite[Th\'eor\`eme 4.7]{WillemsK1}.

\junk{\color{blue}
\begin{eqnarray*}
  a_{wv}^w = \xi_u|_w &=& 
\sum_{R\subset Q\atop \wt\prod R = v}
  \prod_{q\in Q} \left(
  (e^{-\alpha_q})^{[q\notin { R}]}  (1-e^{-\alpha_q})^{[q\in R]}  r_q\right) \cdot 1
\\
  \acirc_{wv}^w = \xi^\circ_u|_w 
  &=& \sum_{R\subset Q\atop \wt\prod R = v}
  \left( \prod_{q\in Q} 
  {(1-e^{-\alpha_q})}^{[q\in R]} \, r_q \, \right) \cdot 1. 
\end{eqnarray*}
}

As an application of Theorem \ref{thm:main}, 
we derive in \S \ref{sec:recursive} a recursive formula for
cohomological structure constants.

\subsection{Properties of the associated operators}

After describing in \S \ref{sec:ingredients} our geometric proof of
Theorem \ref{thm:main}, we give an algebraic interpretation of our
$c_{uv}^w$ formula
as a coefficient of the product of certain ``Schubert structure operators''.  
Let  $H_T^*[\partial]$ denote the smash product of $H_T^*$ with $\ZZ[\partial]$, the algebra consisting of the free $H_T^*$-module 
$H_T^*\otimes_\ZZ \ZZ[\partial]$ with product given by, for $p, q\in H_T^*$,
$$
(p \otimes \partial_v) \cdot (q \otimes \partial_w) = p (\partial_v q) \otimes \partial_v\partial_w
$$
and extended additively. This smash product was used 
by Kostant and Kumar in \cite{KostantKumarH}.
Since $r_\alpha$ acts on $H_T^*(G/B)$ equivalently to $1-\alpha \partial_\alpha$, 
we will abuse notation and denote by $r_\alpha\in H_T^*[\partial]$ the
operator $1-\alpha \partial_\alpha$.
The two actions from \S\ref{ssec:usual} of $\deltop$ on $K_T^*$
lead to two different smash algebras $K_T^*[\Demopisobaric]$ and $K_T^*[\Demop]$,
within which we model $r_\alpha$ 
by $e^\alpha(1 - (1-e^{-\alpha})\Demopisobaric_\alpha)$ or by
$1 - (1-e^{-\alpha})\Demop_\alpha$, respectively.

\begin{Theorem}\label{thm:main2}
Define operators
 $L^\alpha \in H_T^*[\partial] \otimes\ZZ[\partial] \otimes \ZZ[\partial]$,
$\Lambda^\alpha \in K_T^*[\Demop] \otimes \ZZ[\Demop] \otimes \ZZ[\Demop]$,
and \break $\Lambda_\circ^\alpha \in K_T^*[\Demopisobaric] 
\otimes \ZZ[\Demopisobaric] \otimes \ZZ[\Demopisobaric]$ by
$$
\begin{array}{rrrrrrrrr}
  L^\alpha 
  &:=& \partial_\alpha { r_\alpha} \tensor 1 \tensor 1
  &+& r_\alpha \tensor \partial_\alpha \tensor 1
  &+& r_\alpha \tensor 1 \tensor \partial_\alpha 
  &+& \alpha r_\alpha \tensor \partial_\alpha \tensor \partial_\alpha \\
  \Lambda^\alpha
  &:=& (-e^{\alpha}) r_\alpha (-\Demop_\alpha) \tensor 1 \tensor 1
  &+& \ra \tensor \bda \tensor 1  &+& \ra \tensor 1 \tensor \bda
  &-& (1-\ema) \ra \tensor \bda \tensor \bda \\
  \Lambda^\alpha_\circ
  &:=& \ema \ra (-\da)  \tensor 1 \tensor 1
  &+& \ema \ra \tensor \da \tensor 1
  &+& \ema \ra \tensor 1 \tensor \da
  &+& (1-\ema) \ra \tensor \da \tensor \da
\end{array}
  $$
  These \defn{Schubert structure operators} 
  $L^\alpha,\Lambda^\alpha,\Lambda_\circ^\alpha$  
  square to $0,\Lambda^\alpha,\Lambda^\alpha_\circ$ respectively, 
  and obviously commute for orthogonal roots. We prove in \S\ref{se:structureops} that $L^\alpha$ satisfy the braid relation in the simply laced case. By computer we checked
  that $\Lambda^w$ and $\Lambda^w_\circ$ satisfy the appropriate simply-laced braid relation (and
  also checked the doubly-laced relation for $L^\alpha$).

  We may therefore define
  $$
  L^w := \prod_{q\in Q} L^{\alpha_q} \qquad
  \Lambda^w := \prod_{q\in Q} \Lambda^{\alpha_q} \qquad \mbox{and} \qquad
  \Lambda^w_\circ := \prod_{q\in Q} \Lambda^\alpha_\circ
  $$
  for any reduced word $Q$ for $w$ (for $W$ simply laced,
  or even doubly laced in the $L^w$ case).
  In the $\Lambda^w,\Lambda^w_\circ$ cases, it suffices that $Q$
  have Demazure product $w$.
\end{Theorem}

We conjecture that all appropriate braid relations hold. These operators 
act on $H_T^*(G/B)^{\otimes 3}$, $K_T^*(G/B)^{\otimes 3}$, $K_T^*(G/B)^{\otimes 3}$
respectively, resulting in another formulation
(in \S \ref{sec:recursive}) of Theorem \ref{thm:main};
as explained there our formul\ae\ in Theorem \ref{thm:main} were what
motivated us to define the operators of Theorem \ref{thm:main2}.
(Abstractly, the $\Lambda^\alpha,\Lambda^\alpha_\circ$ operators could
be written in terms of the $\deltop$s, but to indicate the
desired actions we used the $\Demopisobaric,\Demop$ notations.)
It seems likely that further analysis of these operators would give a
purely algebraic proof of Theorem \ref{thm:main}.  

There has of course been much previous work on computing these structure
constants for general $G$, even non-positively. In particular \cite{WillemsH} makes use 
of Bott-Samelson manifolds (as do we, in \S\ref{sec:ingredients}) 
to compute point restrictions in both equivariant cohomology and equivariant $K$-theory,
and \cite[Th\'eor\`eme 7.9]{WillemsK2} then uses that formula to compute
structure constants in equivariant $K$-theory using the {\em inverse} of the matrix $\{\xi_w|_v\}_{w,v}$
(much as Billey did in cohomology in \cite[\S5-6]{Billey})
--
in particular, this approach incurs large denominators which cancel,
unlike our manifestly\footnote{%
  Whether it is ``manifest'' to you likely depends on whether you
  {\em really} believe $\partial_\alpha$ of a polynomial $p \in H^*_T$ 
  is again a polynomial. The twisted Leibniz rule
  $\partial_\alpha(pq) = (\partial_\alpha p)q + (r_\alpha p) \partial_\alpha q$
  lets one reduce to the case $p = \beta$ a root, at which point
  $\partial_\alpha \beta = \langle \alpha,\beta \rangle$.}
polynomial approach. We also point out \cite[\S 3]{Duan} (again derived
using Bott-Samelson manifolds), whose formula in (ordinary, non-equivariant) cohomology is quite analogous to the formula $c_{uv}^w = \partial_w(S_u S_v), \ell(w) = \ell(u)+\ell(v)$.
Another formula, based like ours on the multiplication in equivariant cohomology
of the Bott-Samelson manifold, appears in \cite{BerensteinRichmond};
unlike our single-product formula in Theorem \ref{th:b-formula} (cohomology case) for the structure constants of Bott-Samelson calculus, 
their formula \cite[Theorem 2.10]{BerensteinRichmond} requires a sum. 

\section*{Acknowledgments}
Both authors were privileged to hang out with Bert Kostant at MIT 
in the mid-'90s, drinking up an inexhaustible supply of Lie group nectar as well as some history of the academy.
This paper is of course highly inspired by his papers 
\cite{KostantKumarH,KostantKumarK} with Shrawan Kumar.
We thank Bal\'azs Elek for navigating us to the isotropy weights
on Bott-Samelson manifolds, 
Bill Graham for sending us his preprint \cite{Graham}
on the point restrictions $(\xi_w|_v),(\xi^\circ_w|_v)$, and Darij Grinberg
for suggesting an alternate proof of Theorem \ref{thm:main2}.

\section{Ingredients of the proof}\label{sec:ingredients}

Recall that the \defn{Bott-Samelson manifold} associated to a word
$Q = r_{\alpha_{i_1}}r_{\alpha_{i_2}}\cdots r_{\alpha_{i_\ell}}$ in
simple reflections is given by
$$
BS^Q := P_{\alpha_{i_1}} \times^B P_{\alpha_{i_2}} \times^B \cdots
\times^B P_{\alpha_{i_\ell}} /B
$$
where $P_{\alpha_{i_j}}$ is the minimal parabolic associated to the
simple reflection ${r_{i_j}}$ and the quotient is by the
equivalence relation 
\junk{ in $P_{\alpha_{i_1}} \times^B P_{\alpha_{i_2}} \times^B \cdots \times^B
  P_{\alpha_{i_\ell}}$}
given by
$(g_1,g_2, \dots, g_\ell) \sim (g_1b_1,b_1^{-1}g_2b_2, \dots,
b_{\ell-1}^{-1}g_\ell b_\ell).$  The resulting equivalence classes are
denoted $[g_1,g_2,\ldots,g_\ell] \in BS^Q$.


There is an action by $T$ on the left of $BS^Q$ with $2^{\#Q}$ fixed points;
more specifically the set of sequences
$\left\{ (g_1,g_2,\dots, g_\ell) \in
P_{\alpha_{i_1}} \times P_{\alpha_{i_2}} \times \cdots
\times P_{\alpha_{i_\ell}}\ :\ g_j \in \{1,s_j\}\ \forall j\right\}$
maps bijectively to the fixed point set $(BS^Q)^T$.
In this way we index the fixed points by subsets $L\subset \{1,\ldots,\ell\}$,
but instead of writing ``$L$ is the $\{1,2\}$ subword of $(r_2,r_3,r_2)$''
we will write ``$L$ is the subword $r_2 r_3 -$ of $(r_2,r_3,r_2)$'',
allowing distinction between e.g. the $r_2- -$ and $--r_2$ subwords.
The inclusion of the fixed points induces maps in equivariant cohomology 
and in equivariant $K$-theory each of which are known to be injective:
\begin{equation}\label{eq:inclusionBSfixedpoints}
H_T^*(BS^Q) \longrightarrow \bigoplus_{L\subset Q} H_T^*\qquad\qquad
K_T^*(BS^Q) \longrightarrow \bigoplus_{L\subset Q} K_T^*.
\end{equation}

\junk{ 
We provide some explanation of notation. For simplicity, we define $s_j = r_{i_j}$ and denote
$Q=s_1s_2\dots s_\ell$. Then each subword $L=s_{t_1}s_{t_2}\cdots s_{t_k}$  of $Q$ corresponds to a subset $\{t_1,\dots, t_k\}$ of $\{1,\dots, \ell\}$, which in turn indicates an ordered subset of possibly repeating simple reflections (those occurring as $s_{t_1},\cdots, s_{t_k}$), together with their positions within $Q$. We abuse notation and write $L\subset Q$ to indicate the ordered set of (possibly repeated) simple reflections occurring in positions $t_1,\dots, t_k$. Then for $t_b\in \{t_1,\dots, t_k\}$, and $s_{t_b}=r_m$, we simply write $m\in L\subset Q$ where the position has been suppressed.  Similarly, we write $\alpha_m$ for the simple root corresponding to the reflection.  While the same simple reflection $r_m$ may occur more than once in $L$, the corresponding position does not repeat. Furthermore, we indicate by $
\prod_{m\in L}$ the product that concatenates in the order of the subword $L$. For example, if $Q = r_1r_2r_1$ and $L = r_1 r_2$, then 
$$
\prod_{m\in L} \alpha_m r_m = \alpha_1r_1\alpha_2r_2.
$$
}

For any subword $L=s_{t_1}\cdots s_{t_k}$ of $Q$, there is a
corresponding copy of $BS^L$ obtained as a submanifold of $BS^Q$ by
$$
BS^L = \left\{ [g_1,\cdots, g_\ell]\in BS^Q\  |\  g_j\in B \mbox{ if  $j\not\in L$} \right\}.
$$
The submanifolds $BS^L$ are $T$-invariant, and each $BS^L_\circ := BS^L \setminus \bigcup_{M \subsetneq L} BS^M$ contains a unique $T$-fixed point 
$[g_1, \dots, g_\ell]\in BS^L$, the one we also corresponded to $L$. 

The equivariant homology classes $\{[BS^L] \ :\ L\subset Q\}$ form a
basis of $H_*^T(BS^Q)$ as a (free) module over $H_T^*$. There exists a
dual basis $\{T_J\}_{J\subset Q}$ of $H^*_T(BS^Q)$, again defined by
the $H^*_T$-valued Alexander pairing $\langle,\rangle$.  The classes
$\{[\calO_{BS^L}]\ : \ L\subset Q\}$,
$\{[I_{BS^L}]\ : \ L\subset Q\}$ also form bases for $K^T_*(BS^Q)$ as
a module over $K_T^*$, again with the simple change-of-basis matrices
$$ [\calO_{BS^R}] = \sum_{S\subset R} [I_{BS^S}], \qquad\qquad
[I_{BS^R}] = \sum_{S\subset R} (-1)^{|R\setminus S|}[\calO_{BS^S}]. $$
We denote the dual bases under the Alexander pairing 
by $\{\tau^\circ_J\}_{J\subset Q}, \{\tau_J\}_{J\subset Q}$ respectively, with the result that
\begin{equation}\label{eq:tautaucirc}
\tau^\circ_R = \sum\limits_{P\supset R} (-1)^{|P\setminus R|} \tau_P, \qquad\qquad \tau_R = \sum\limits_{P\supset R} \tau^\circ_P
\end{equation}
We compute the point restrictions of $T_J$, $\tau_J^\circ$, $\tau_J$
in Lemma \ref{lemma:Trestriction}.

Consider the natural map $\pi_R:\ BS^R \rightarrow G/B$ that multiplies
the terms, $[g_1,\ldots,g_\ell] \mapsto (\prod_i g_i) B/B$. 
The image is $B$-invariant, irreducible, and closed,
so is necessarily some $X^w$ (indeed, $w$ is the Demazure product $\wt\prod R$).
However $\dim BS^R = \dim X^w$ if and only if $R$ is a reduced word, 
in which case the top homology class of $BS^R$ pushes forward to that of $X^w$. 
\junk{ Furthermore, for $I$ a reduced word for the longest word $w_0$, the map $\pi_I: BS^I\rightarrow G/B$ is surjective with $\pi_I(BS^Q\subset BS^I)=X^w$ whenever $w=\prod Q$. {\bf Who cares}}
The pushforward sends the homology class of $BS^R$ to that
of $X^w$ in $G/B$ whenever $R$ is a reduced word for $w$, and
otherwise sends it to $0$. These statements are true both for singular
homology and also, since the varieties involved are $T$-invariant,
for 
equivariant homology \cite{KumarNori,BrionH}.  
For the corresponding statement in equivariant $K$-theory, 
note that the pushforward at the level of {\em
  sheaves} is $\pi_*(\calO_{BS^R})=\calO_{X^w}$
($w = \wt\prod R$)
with no higher derived pushforwards \cite[Theorem 3.4.3]{Brion-Kumar}, 
so of course it's also true at the level of
$K$-theory classes. For lack of a reference, we compute $\pi_*[I_{BS^R}]$:

\begin{Lemma}
  $ \pi_*[I_{BS^R}] = (-1)^{|R|-\ell(w)}\ [I_{X^{w}}] $, where $w = \wt\prod R$.
\end{Lemma}

\begin{proof} We write $\calO(Y)$ for $\calO_Y$ in this proof, for
  readability.
  \begin{eqnarray*}
&&     \pi_*[I_{BS^R}] \\
    &=& \pi_* \sum_{S\subset R} (-1)^{|R\setminus S|} [\calO(BS^S)] 
    = \sum_{S\subset R} (-1)^{|R\setminus S|} \pi_* [\calO(BS^S)] 
    = \sum_{S\subset R} (-1)^{|R\setminus S|} [\calO(X^{\wt\prod S})] \\
    &=& \sum_{u \leq \wt\prod R} [\calO(X^u)] 
        \sum_{S\subset R, \ \wt\prod S = u} (-1)^{|S\setminus R|} 
    = \sum_{u \leq \wt\prod R} [\calO(X^u)]\ (-1)^{|R|} 
        \sum_{S\subset R, \ \wt\prod S = u} (-1)^{|S|} \\
    &=& \sum_{u \leq \wt\prod R} [\calO(X^u)]\ (-1)^{|R|} \
        \sum \left\{ (-1)^{|S|} \ :\ R\setminus S\text{ an interior face of 
        $\Delta(R,u)$} \right\} \\
    &=& \sum_{u \leq \wt\prod R} [\calO(X^u)] \
        \phantom{(-1)^{|R|} }
        \sum \left\{ (-1)^{|F|} \ :\ \phantom{R\setminus S} F \text{ an interior face of 
        $\Delta(R,u)$} \right\} \\
    &=& \sum_{u \leq \wt\prod R} [\calO(X^u)] \left(
        \left[u = \wt\prod R\right] +
        \sum \left\{ (-1)^{|F|} \ :\ F\neq \emptyset, 
        F \text{ an interior face of $\Delta(R,u)$} \right\} \right) \\
    &=& \sum_{u \leq \wt\prod R} [\calO(X^u)] \left(
        \left[u = \wt\prod R\right] +
        \sum \left\{ -\chi_c(\text{open $(|F|-1)$-simplex}) \ :\ 
        F\neq\emptyset, F \text{ face of $\Delta(R,u)^\circ$} \right\}  \right)\\
    &=& \sum_{u \leq \wt\prod R} [\calO(X^u)]\  
        \left(\ \left[u = \wt\prod R\right] - \chi_c(\Delta(R, u)^\circ)  \right)
        \qquad\text{by additivity of $\chi_c$}
  \end{eqnarray*}
  where $\Delta(R,u)$ is the subword complex of the pair $(R,u)$
  introduced in \cite{KMsubword}, proven to be a $(|R|-\ell(u)-1)$-ball for
  $u < \wt\prod R$, and a $(|R|-\ell(u)-1)$-sphere for $u = \wt\prod R$. 
  The signed sum is computing the negative of the Euler characteristic
  of the {\em interior} of this ball/sphere (because of the
  $\wt\prod S = u$ not just $\geq u$ condition), plus $1$ in the
  sphere case: the topology never sees the empty face, which is okay exactly
  in the ball case (since then and only then,
  the empty face doesn't lie in the interior).

  Consequently, the parenthetical term is $(-1)^{|R|-\ell(u)}$, for different
  reasons in the ball, even-dim sphere, and odd-dim sphere cases,
  and so we continue
  $$ =  \sum_{u \leq \wt\prod R} [\calO(X^u)]\  (-1)^{|R|-\ell(u)}
  = (-1)^{|R|-\ell(w)} 
  \sum_{u \leq w} [\calO(X^u)]\  (-1)^{\ell(w)-\ell(u)}
  = (-1)^{|R|-\ell(w)} [I_{X^{w}}]
  $$
  where $w = \wt\prod R$.
\end{proof}

We are interested in the resulting transpose maps in equivariant cohomology and
equivariant $K$-theory. We keep track of the various dual bases here:
$$
\begin{array}{ccccc}
  \underline{\text{($K$-)homology basis}} 
  &\text{dual to} & \underline{\text{($K$-)cohomology basis}}
  &\text{with} & \underline{\text{structure constants}} \\
  \\ \\
  \begin{array}{c|c|c}
    & G/B & BS^Q \\
    \hline   H_* & [X^w] & [BS^J] \\
    \hline   K_* & [\calO_{X^w}] & [\calO_{BS^J}] \\
    \hline   K_* & [I_{X^w}] & [I_{BS^J}] \\
  \end{array} 
  &&
     \begin{array}{c|c|c}
       & G/B & BS^Q \\
       \hline     H^* & S_w & T_J \\
       \hline     K^* & \xi^\circ_w & \tau_J^\circ \\
       \hline     K^* & \xi_w & \tau_J \\
     \end{array}
  &&
     \begin{array}{c|c|c}
       & G/B & BS^Q \\
       \hline     H^* & c_{uv}^w & b_{uv}^w \\
       \hline     K^* & \acirc_{uv}^w & \dcirc_{uv}^w \\
       \hline     K^* & a_{uv}^w & d_{uv}^w 
     \end{array}
\end{array}
$$
\junk{
\AKrem{The below was stated all in cohomology instead of, properly,
  half in homology, so I made the cube above. Maybe we'll have another line
  for the other bases of $K$}
$$
\begin{array}{ccccc}
 \{T_J\} &\mbox{ dual to } &\{[BS^J]\}& \mbox{ in } &H^*_T(BS^Q),\\ 
 \{S_w\} &\mbox{ dual to } & \{S^w\} &\mbox{ in } &H^*_T(G/B),\\
\{\tau_J^\circ\} &\mbox{ dual to } &\{[\calO_{BS^J}]\} &\mbox{ in } &K_T^*(BS^Q),\\ 
\{\xi_w\} &\mbox{ dual to } &\{[\calO_{X^w}]\} &\mbox{ in } &K_T^*(G/B)
\end{array}
$$
}
\junk{
  As with ordinary cohomology, there is a pairing between the
  equivariant homology $H^T_*(G/B)$  and the equivariant cohomology $H_T^*(G/B)$. The $H^*_T$-basis of equivariant homology classes associated with $\{X^w\}$ has a dual basis  $\{x_w\}\in H_T^*(G/B)$ given by the relationship
  $$
  \int_{X^w} x_v =\delta_{wv},
  $$
  where integration is shorthand notation for capping the equivariant homology element corresponding to the $T$-invariant variety $X^w$ with the equivariant cohomology class $x_v$, then pushing forward to a point \cite{Graham}.  A minor adaption of an argument in that work shows that $x_w = S_w$ is the opposite Schubert class (associated with the opposite Schubert variety $X_w$).
}

Transposing the statements 
$(\pi_Q)_*([BS^R]) = [X^w] [R$ reduced$]$, 
$(\pi_Q)_*([\calO_{BS^R}]) = [\calO_{X^w}]$, 
$(\pi_Q)_*([I_{BS^R}]) = (-1)^{|R|-\ell(w)}[I_{X^{w}}]$ (for $w = \wt\prod R$)
from ($K$-)homology, we obtain the following:

\junk{Ooh! We could have a factor $\beta^{|R|-\ell(w)}$, with
  $\beta = 0,1,-1$ in the respective cases!}

\begin{Lemma} 
Let $\pi_Q: BS^Q \rightarrow G/B$ be the product map. 
$$
\pi_Q^*(S_w) = \sum_{R\subset Q\text{ reduced} \atop  \prod R =w} T_R,
\qquad
\pi_Q^*(\xi^\circ_w) = \sum_{R\subset Q \atop \wt\prod R = w}\tau_R^\circ,
\qquad\text{and}\qquad
\pi_Q^*(\xi_w) = \sum_{R\subset Q \atop \wt\prod R = w} (-1)^{|R|-\ell(w)} \tau_R
$$
in equivariant cohomology and equivariant $K$-theory.
\end{Lemma}

\begin{proof} 
 The result $\pi_Q^*(S_w)$ is  \cite[Proposition 3.26]{WillemsH}). 
  Let $[BS^R], [X^w]$ denote the equivariant homology classes, 
  and $\langle,\rangle_M$ denote the perfect $H^*_T$-valued pairing
  between $H^T_*(M)$ and $H_T^*(M)$ for $M$ a smooth compact
  oriented $T$-manifold.  Then
  \begin{eqnarray*}
    \langle \pi_Q^*(S_w), [BS^R]\rangle_{BS^Q} &=&
  \langle S_w, (\pi_Q)_*( [BS^R]) \rangle_{G/B} 
   =  \begin{cases}
    \langle S_w, [X^v] \rangle_{G/B}
    & \mbox{if $R$ reduced, with $\prod R = v$}\\
    0 & \mbox{otherwise}
  \end{cases} \\
  &=&  \begin{cases}
    1
    & \mbox{if $R$ reduced, with $\prod R = w$}\\
    0 & \mbox{otherwise.}
  \end{cases}
  \end{eqnarray*}
  
\junk{
  Similarly, by using the $K$-theoretic pairing
  $\langle, \rangle^K_{M}$ on $M$, we have
   \begin{eqnarray*}
    \langle \pi_Q^*(\xi_w), [\calO_{BS^L}]\rangle^{K}_{BS^Q} &=&
  \langle \xi_w, (\pi_Q)_*( [BS^L]) \rangle^{K}_{G/B} 
   =  \begin{cases}
    \langle \xi_w, [\calO_{X^v}] \rangle^K_{G/B} 
    & \mbox{if $\wt\prod L = v$}\\
    0 & \mbox{otherwise}
  \end{cases} \\
  &=&  \begin{cases}
    1
    & \mbox{if $\wt\prod L = w$}\\
    0 & \mbox{otherwise.}
  \end{cases}
  \end{eqnarray*}

\junk{  This pairing is $1$ when $v=w$, and $0$ otherwise. 
  Therefore, $\langle \pi_I^*(S_w), [BS^L] \rangle=1$ if and only if $L$ is a
  reduced word for $w$. }
  
Since  $\{T_R\}$ are defined so that
  $\langle T_R, [BS^L] \rangle_{BS^Q} = \delta_{RL}$, and $\{\tau_R^\circ\}$ so that $\langle \tau_R^\circ, [\calO_{BS^L}]\rangle^K_{BS^Q}=\delta_{RL}$, we conclude that
  $\pi_Q^*(S_w) = \sum_{R\subset Q\text{ reduced} \atop \prod R=w} T_R$ and $\pi_Q^*(\xi_w) = \sum_{R\subset Q \atop \wt\prod R=w} \tau^\circ_R$.
}
The other two statements are similarly tautological.
\end{proof}

We pull back the equation
$S_u S_v = \sum_{x\in W} c_{uv}^x S_x$ along $\pi_Q:\ BS^Q\to G/B$ 
and simplify the right-hand side of the equation:
\begin{equation}
  \label{eq:LHS}
  \pi_Q^*(S_u)\ \pi_Q^*(S_v) 
=  \sum_{x\in W} c_{uv}^x\ \pi_Q^*(S_x) 
= \sum_{x\in W} c_{uv}^x\sum_{R\subset Q\text{ reduced} \atop \prod R =x} T_R
= \sum_{R\subset Q\text{ reduced}}  c_{\ uv}^{\prod R}\  T_R.
\end{equation}
By expanding the left hand side in a similar fashion, we obtain
$$
\pi_Q^*(S_u) \pi_Q^*(S_v) =  \sum_{R\subset Q \text{ reduced} \atop \prod R =u} T_R \sum_{S\subset Q\text{ reduced}\atop \prod S =v} T_S
= \sum_{R,S\subset Q\text{ reduced} \atop \prod R =u, \prod S = v} T_R T_S.
$$
Define $b_{RS}^J$ to be the structure constants for the multiplication
in $H_T^*(BS^Q)$ in the basis $\{T_J\}$, defined by the relationship
$$
 T_R T_S = \sum_{J\subset Q} b_{RS}^J T_J.
$$
Thus we have shown
\begin{equation}
  \label{eq:RHS}
\pi_Q^*(S_u) \pi_Q^*(S_v) = \sum_{R,S\subset Q\text{ reduced} \atop  \prod R =u, \prod S = v}  \sum_{J\subset Q} b_{RS}^J T_J.
\end{equation}
Now take $Q$ to be reduced with $\prod Q = w$ 
and match the $T_Q$ coefficients in (\ref{eq:LHS}) and (\ref{eq:RHS}):
\begin{equation}\label{eq:cs=sumofbs}
  c_{uv}^w = \sum_{R,S\subset Q \text{ reduced} \atop  \prod R =u,\, \prod S = v} b_{RS}^Q.
\end{equation}
The $K$-theoretic results are similar. Let $Q$ be any expression whose
Demazure product is $w$, e.g. a reduced word. We obtain by the same derivation
\begin{equation}
  \label{eq:as=sumofds}
  a_{uv}^w = \sum_{R,S\subset Q\atop \wt\prod R = u,\, \wt\prod S = v} 
  (-1)^{|R|+|S|-\ell(u)-\ell(v)} d_{RS}^Q\ ,
  \qquad\qquad
  \acirc_{uv}^w = \sum_{R,S\subset Q\atop \wt\prod R = u,\, \wt\prod S = v} \dcirc_{RS}^Q
\end{equation}
where the sum is over all subwords whose Demazure products 
are $u$ and $v$, respectively.

\begin{Theorem}\label{th:b-formula} Let the equivariant intersection
  numbers $b_{RS}^Q$ be defined as in \eqref{eq:cs=sumofbs}. For $R,S\subset Q$,
  $$
  b_{RS}^Q  \quad = \quad
   \prod_{q\in Q} \left( \alpha_q^{[q\in R \cap S]} \partial_q^{\left[q\notin R\cup S\right]}
    r_q \right)\cdot 1 \quad = \quad
   \prod_{q\in Q} \left( \alpha_q^{[q\in R\cap S]} r_q (-\partial_q)^{\left[q\notin R\cup S\right]}
    \right)\cdot 1,
  $$
  where the exponent $[q\in R\cap S] \in \{0,1\}$ (resp. $[q\notin R\cup S]$) indicates inclusion of the
  factor only when $q\in R\cap S$ (resp. q$\notin R\cup S$). Similarly, for $d_{RS}^Q$ and $\dcirc_{RS}^Q$ defined in \eqref{eq:as=sumofds},
  \begin{eqnarray*}
    d_{RS}^Q &=&\left( \prod_{q\in Q} {(e^{\alpha_q})^{ [q\notin {R\cup S}] }}
  {(1-e^{-\alpha_q})}^{[q\in R \cap S]} \, r_q \, 
                 (-\Demop_q)^{[q\notin  R \cup S]}) \, \right) \cdot 1 
                 \qquad
    \text{where }\ \Demop_q f := \frac{f - r_\alpha f}{1-e^{-\alpha}} \\
  \dcirc_{RS}^Q &=& \left( \prod_{q\in Q} 
  (e^{-\alpha_q})^{[q\notin {R\cap S}]} (1-e^{-\alpha_q})^{[q\in R \cap S]} 
  r_q \,  (-\Demopisobaric_q)^{[q\notin R \cup S]}  \right) \cdot 1 
                    \qquad
    \text{where }\ \Demopisobaric_q f := \frac{f - e^{-\alpha}r_\alpha f}{1-e^{-\alpha}}\\
  \end{eqnarray*}
\junk{
  \AKrem{old version:}
  $$
  d_{RS}^Q = \bigg( \prod_Q \,xs
  (e^{-\alpha_q})^{[q\in \overline{V\cap W}]}(1-e^{-\alpha_q})^{[i\in V \cap W]}  (-\delta_q)^{[q\in \bar V \cap \bar W]} \, r_q^{q\in V\cup W} \bigg) \cdot 1
  $$
}
\junk{
  where the $\Demopisobaric_i$ are the usual \defn{isobaric divided difference operators} in
  the simple roots $\alpha_i$,
  defined by $\Demopisobaric_i f := (f - e^{-\alpha_i}r_i\cdot f)/ (1-e^{-\alpha_i})$.
}
\end{Theorem}
\junk{
  {\color{blue} ``Do we need to concern ourselves with what we mean when $R,S$ are not reduced words contained in $Q$?'' If we want $H^*(BS^Q)$ then yes. While we show that $b_{RS}^Q=0$ when $Q$ doesn't contain $R,S$, we never need an interpretation for $b_{RS}^Q$ when $R,S,Q$ are not reduced. They may not be zero.}
  \AKrem{Where do we use that $R,S$ are reduced anyway?}
}

Theorem \ref{thm:main} then follows directly from
Theorem~\ref{th:b-formula} and \eqref{eq:cs=sumofbs},\eqref{eq:as=sumofds}.
The proof of Theorem~\ref{th:b-formula} is an inductive argument based on Lemma \ref{lemma:Trestriction} 
below.

As with Schubert classes, we define the point restriction $T_R|_S$ to
be the pullback of $T_R \in H_T^*(BS^Q)$ along the
inclusion of 
the fixed point $S\subset Q$. 
These restrictions can be computed explicitly. 
In the following Lemma, the results for $T_R|_S$ and $\tau_R^\circ|_S$ (first equality) are found in  \cite[Th\'eor\`eme 3.11]{WillemsH}) and \cite[Th\'eor\`eme 6.2]{WillemsK2}, respectively.
\begin{Lemma}\label{lemma:Trestriction}
  The classes $T_R \in H_T^*(BS^Q), \tau_R \in K_T^*(BS^Q)$ 
  have the following restrictions to a $T$-fixed point $S$:
  $$
  T_R|_S=\begin{cases}
    \displaystyle \left(\prod_{m\in S} \alpha_{m}^{[m\in R]} r_m \right)\cdot 1 &\mbox{if $R\subset S$}\\
    0 &\mbox{if $R \not\subset S$}
  \end{cases}
\qquad
  \tau_R |_S =\begin{cases}
    \displaystyle
    \left(\prod_{m\in S} (1-e^{-\alpha_m})^{[m\in R]} r_m\right)\cdot 1  
&\mbox{if $R\subset S$}\\
    0 &\mbox{if $R \not\subset S$.}
  \end{cases}
  $$
  where the exponent $[m\in R]$ indicates inclusion of the factor only
  when $m\in R$.
 The classes $\tau_R^\circ\in K_T^*(BS^Q)$ have the following restrictions
  to fixed points:
  $$
  \tau^\circ_R |_S 
  =\begin{cases}
    \displaystyle
    \left(\prod_{m\in S} (e^{-\alpha_m})^{[m\notin R]} (1-e^{-\alpha_m})^{[m\in R]} r_m\right)\cdot 1  
    &\mbox{if $R\subset S$}\\
    0 &\mbox{if $R \not\subset S$}
  \end{cases}
    \qquad =  \sum_{J\subset S} (-1)^{|J\setminus R|}
    \tau_J|_{S} 
    $$
  where we note that the last sum is $0$ if $R$ is not contained in $S$.
In particular none of $T_R|_R$, $\tau_R|_R$, and $\tau^\circ_R|_R$ vanish.
  \end{Lemma}

\junk{Confirmation of the last equality, assuming $J\subset L$:
  \begin{eqnarray*}
  \tau^\circ_J |_S 
  &=& \left(\prod_{m\in L} e^{-\alpha_m [m\notin J]} (1-e^{-\alpha_m})^{[m\in J]} r_m\right)\cdot 1  \\
  &=& \left(\prod_{m\in L} (1-(1-e^{-\alpha_m}))^{[m\notin J]} (1-e^{-\alpha_m})^{[m\in J]} r_m\right)\cdot 1  \\
&=&    \sum_{R:\ J\subset R \subset L} 
\prod_{m\in L} \left( (-(1-e^{-\alpha_m}))^{[m \in R\setminus J]} (1-e^{-\alpha_m})^{[m\in J]} r_m\right)\cdot 1  \\
&=&    \sum_{R:\ J\subset R \subset L} (-1)^{|R\setminus J|}    
\prod_{m\in L} \left( (1-e^{-\alpha_m})^{[m \in R\setminus J]} (1-e^{-\alpha_m})^{[m\in J]} r_m\right)\cdot 1  \\
&=&    \sum_{R:\ J\subset R \subset L} (-1)^{|R\setminus J|}    
\prod_{m\in L} \left( (1-e^{-\alpha_m})^{[m \in R]} r_m\right)\cdot 1  
=    \sum_{R:\ J\subset R \subset L} (-1)^{|R\setminus J|}  \tau_J|_R
  \end{eqnarray*}
}

In proofs we will make use of the {plainly} equivalent formula
\begin{equation*}
\tau_R^\circ |_S = \sum_{J:\ R\subset J \subset S} 
(-1)^{|J\setminus R|} \prod_{t\in J} \left(1-e^{\left(\prod_{m\in S, m\leq t} r_m\right)\alpha_t}\right)
\end{equation*}

\junk{
It follows that, for a fixed reduced word $Q$ for $w$, 
$$
c_{uv}^w =         
\sum_{R,S \subset Q\text{ reduced}\atop \prod R = u, \prod S = v}
\prod_Q \left( \alpha_q^{[q\in R, S]} \partial_q^{\left[q\notin R, S\right]} r_q \right)\cdot 1
$$
A quick comparison to $d_{uv}^w$ shows that $c_{uv}^w = d_{uv}^ww \cdot 1$. The statement of the theorem follows.
}

The proof of this Lemma~\ref{lemma:Trestriction} and Theorem~\ref{th:b-formula} are left to \S\ref{se:Restrctionproof}.

\junk{
We present these coefficients in terms of some apparently natural
families of operators, based on reflections and divided difference operators.
}

\section{AJS/Billey operators}\label{sec:ajsb}

In the next two sections we interpret the AJS/Billey formula,
and Theorem \ref{thm:main}, in terms of certain operators;
our results are that these operators satisfy the various (nil-)Coxeter 
relations. We hope someday to run the arguments backward and use the relations
to give an algebraic proof of Theorem \ref{thm:main}.

We recall the usual statement \cite{AJS,Billey}. Let $Q$ be a word 
in $W$'s generators, with product $w$ (asked in \cite{Billey} to be
reduced, though this is not necessary and is not demanded in 
\cite[Appendix D]{AJS}). Let $\alpha_1,\ldots,\alpha_{\#Q}$ be the
corresponding sequence of simple roots.  
Define an associated sequence of roots
$$ \beta_i := r_{\alpha_1} r_{\alpha_2} \cdots r_{\alpha_{i-1}} \cdot \alpha_i $$
which, {\em if} $Q$ happens to be reduced, are positive roots. 
Then their formula states
$$ S_v|_w 
= \sum \left\{ \prod_R \beta_r\ :\ R\subset Q,\ R
  \text{ is a reduced word for $v$}\right\} 
$$
As a first step in reformulating this operator-theoretically, 
we rewrite it as
$$ S_v|_w 
= \sum_R \left( \prod_Q { \alpha_q}^{[q\in R]} r_q \right) \cdot 1
\qquad \text{again summing over reduced words $R$ for $v$}
$$
To actually evaluate such a formula involves dragging all the reflection
operators $r$ to the right of all the multiplication operators $\alpha$,
thereby turning each $\alpha$ into the corresponding $\beta$ from the
associated sequence. There those reflections multiply back up to $w$,
which {\em without} the action $\cdot\, 1$ gives
an operator equation we record for use below:
\begin{equation}
  \label{eq:AJSBop}
  \sum_R \prod_Q { \alpha_q}^{[q\in R]} r_q
  = S_v|_w\ w
\qquad \text{again summing over reduced words $R$ for $v$}
\end{equation}

\subsection{The operators}\label{ssec:AJSBops}

Let $H^*_T[W]$ be the smash product of $H_T^*$ and the group algebra of $W$, i.e. the free $H^*_T$-module with basis $W$ and multiplication $wp := (w\cdot p) w$.
For each $w\in W$, we introduce an \defn{AJS/Billey operator} 
\begin{equation}
  \label{eq:Jdef}
  \begin{array}{rccc}
    
  J_w := \sum_{v} (S_v|_w) w \ \tensor\ \partial_v 
  \quad \in& H^*_T[W] \tensor_{\ZZ} \ZZ[\partial] 
    &\into& H^*_T[\partial] \tensor_{\ZZ} \ZZ[\partial] \\
    & r_\alpha \tensor p &\mapsto& (1-\alpha \partial_\alpha) \tensor p
  \end{array}
\end{equation}
so in particular
$$ J_\alpha := J_{r_\alpha} = 
(r_\alpha\tensor 1)+(\alpha r_\alpha\tensor\partial_\alpha). $$
Note that these operators are homogeneous of degree $0$, where
the degrees of $\alpha,r_\alpha,\partial_\alpha$ are $+1,0,-1$ respectively.

\junk{
\begin{Theorem}\label{thm:AJSBops}
  \begin{enumerate}
  \item If $Q$ is a reduced word for $w$, then $J_w = \prod_Q J_q$.
  \item If $\ell(w)+\ell(v) = \ell(wv)$, then $J_w J_v = J_{wv}$,
    and this fact is essentially equivalent to the AJS/Billey formula.
  \item $J_\alpha^2 = 1\tensor 1$, so in fact any word $Q$ for $w$ suffices
    in (1), and $J_w J_v = J_{wv}$ for all $w,v$.
  \end{enumerate}
\end{Theorem}
}

\begin{Theorem}\label{thm:AJSBops}
  If $Q$ is a word for $w$, then $J_w = \prod_Q J_q$. Hence
  ${J_\alpha}^2 = 1\tensor 1$, and $J_w J_v = J_{wv}\ \forall w,v$.
\end{Theorem}

\begin{proof}
  Let $Q$ be a word for $w$. Then since $S_v|_{r_\alpha}$ is $0$ unless $v=1$ or $v=r_\alpha$,
    \begin{eqnarray*}
      \prod_Q J_q 
      &=& \prod_Q \sum_{v } (S_v|_{r_q}) r_q \tensor \partial_q 
      \quad=\quad \prod_Q 
          \left( (r_q \tensor 1) + (\alpha_q r_q \tensor \partial_q) \right) \\
      &=& \sum_{R\subset Q} \left(\prod_Q \alpha_q^{[q\in R]} r_q\right) 
          \tensor \prod_R \partial_r 
      \quad=\quad 
          \sum_v \sum_{R\subset Q\text{ reduced}\atop \prod R=v} 
                    \left(\prod_Q \alpha_q^{[q\in R]} r_q\right) 
          \tensor \partial_v
     \end{eqnarray*}         
     as $\prod_R \partial_r=0$ unless $R$ is reduced.
     Using equation (\ref{eq:AJSBop}), this becomes
     $\sum_v (S_v|_w\ w) \tensor \partial_v$ or $J_w$.

     Then ${J_\alpha}^2 = J_e = 1\tensor 1$ (which one can check directly),
     and the last claim follows by concatenating words for $w,v$.
\end{proof}

\junk{
\begin{proof}
  \begin{enumerate}
  \item Let $Q$ be a reduced word for $w$. Then since $S_v|_{r_\alpha}$ is $0$ unless $v=1$ or $v=r_\alpha$,
    \begin{eqnarray*}
      \prod_Q J_q 
      &=& \prod_Q \sum_{v } (S_v|_{r_q}) r_q \tensor \partial_q 
      \quad=\quad \prod_Q 
          \left( (r_q \tensor 1) + (\alpha_q r_q \tensor \partial_q) \right) \\
      &=& \sum_{R\subset Q} \left(\prod_Q \alpha_q^{[q\in R]} r_q\right) 
          \tensor \prod_R \partial_r 
      \quad=\quad 
          \sum_v \sum_{R\subset Q\text{ reduced}\atop \prod R=v} 
                    \left(\prod_Q \alpha_q^{[q\in R]} r_q\right) 
          \tensor \partial_v
     \end{eqnarray*}         
as $\prod_R \partial_r=0$ unless $R$ is reduced. The AJS/Billey formula 
 directly implies that
$$
\sum_{R\subset Q\text{ reduced}\atop \prod R=v} \prod_Q \alpha_q^{[q\in R]} r_q
= \sum_{R\subset Q\text{ reduced}\atop \prod R=v}
\left(\prod_{q\in R} \left(\prod_{q'\text{ left of }q} r_{q'}\right)
  \cdot \alpha_q\right) \prod_Q r_q
= S_v|_w\, w,
$$
(as these $\prod_{q\in R} \left(\prod_{q'\text{ left of }q} r_{q'}\right)
  \cdot \alpha_q$ are Billey's $\beta_q$)
from which it follows that 
$$
\prod_Q J_q =  \sum_{v} (S_v|_w) w \tensor \partial_v=J_w.
$$

  \item From (1) the equality $J_w J_v = J_{wv}$  follows
   by concatenating words for $w$ and $v$. \\
    Conversely, the equality implies $J_w = \prod_Q J_q$ when $Q$ is a reduced word for $w$,
    which in turn implies the AJS/Billey formula by the calculation above.
  \item 
    \begin{eqnarray*}
      J_\alpha^2 
  &=& \left( (r_\alpha\tensor 1)+(\alpha r_\alpha\tensor\partial_\alpha)\right)^2
  = \left( (r_\alpha\tensor 1)+(\alpha r_\alpha\tensor\partial_\alpha)\right)
      \left( (r_\alpha\tensor 1)+(\alpha r_\alpha\tensor\partial_\alpha)\right)\\
  &=& (1\tensor 1) 
      + (r_\alpha \alpha r_\alpha \tensor \partial_\alpha)
      + (\alpha \tensor \partial_\alpha)
      + ( \alpha r_\alpha  \alpha r_\alpha \tensor \partial_\alpha^2) = 1\tensor 1
    \end{eqnarray*}
  \end{enumerate}
\end{proof}
}

The corresponding definitions and results in $K$-theory are very similar:

$$ \Xi_w := \sum_{v}  (-1)^{\ell(v)} (\xi_v|_w)\, w \tensor \Demop_v, \qquad
\Xi^\circ_w := \sum_{v} (\xi^\circ_v|_w)w \tensor \Demopisobaric_v $$
(the latter operators are same, 
up to $\delta_\alpha \leftrightarrow r_\alpha \delta_\alpha r_\alpha$
in the second tensor slot, the operators $L(w)$ from 
\cite[Proposition 3.6]{Graham})
so in particular
$$ \Xi_\alpha := \Xi_{r_\alpha} = 
(r_\alpha \tensor 1) 
- ((1-e^{-\alpha})r_\alpha \tensor \Demop_\alpha)
\qquad
\Xi^\circ_\alpha := \Xi^\circ_{r_\alpha} =
(e^{-\alpha} r_\alpha \tensor 1) 
+ ((1-e^{-\alpha})r_\alpha \tensor \Demopisobaric_\alpha)
$$
and when $Q$ is a reduced word for $w$, the Graham/Willems formula gives 
\begin{eqnarray*}
  \prod_Q \Xi_q 
  &=& \prod_Q \left((r_q \tensor 1) -((1-e^{-\alpha_q})r_q \tensor \Demop_q)\right)
  = \sum_{R\subset Q} \left(\prod_Q (-1)^{[q\in R]} (1-e^{-\alpha_q})^{[q\in R]}r_q \right)
      \tensor \prod_R \Demop_r \\
  &=& \sum_v \sum_{R\subset Q\atop \wt\prod R=v}
      \left(\prod_Q(-1)^{[q\in R]} (1-e^{-\alpha_q})^{[q\in R]}r_q \right)
      \tensor \Demop_v
      = \sum_v (-1)^{\ell(v)} (\xi_v|_w) w\tensor \Demop_v = \Xi_w
\end{eqnarray*}
with much the same derivation for $\prod_Q \Xi^\circ_q = \Xi^\circ_w$,
which we omit.
For variety we check $(\Xi^\circ_q)^2 = 1\tensor 1$:
\begin{eqnarray*}
  (\Xi^\circ_q)^2 
  &=& \left( (e^{-\alpha} r_\alpha \tensor 1) 
      + ((1-e^{-\alpha})r_\alpha \tensor \Demopisobaric_\alpha) \right)^2 \\
  &=& \left( (e^{-\alpha} r_\alpha \tensor 1) 
      + ((1-e^{-\alpha})r_\alpha \tensor \Demopisobaric_\alpha) \right)
      \left( (e^{-\alpha} r_\alpha \tensor 1) 
      + ((1-e^{-\alpha})r_\alpha \tensor \Demopisobaric_\alpha) \right) \\
  &=& (e^{-\alpha} r_\alpha \tensor 1) (e^{-\alpha} r_\alpha \tensor 1) 
+ ((1-e^{-\alpha})r_\alpha \tensor \Demopisobaric_\alpha) (e^{-\alpha} r_\alpha \tensor 1) \\
&&+ (e^{-\alpha} r_\alpha \tensor 1) ((1-e^{-\alpha})r_\alpha \tensor \Demopisobaric_\alpha) 
+ ((1-e^{-\alpha})r_\alpha \tensor \Demopisobaric_\alpha) 
((1-e^{-\alpha})r_\alpha \tensor \Demopisobaric_\alpha) \\
  &=& (1\tensor 1) + (e^{\alpha} - 1) \tensor \Demopisobaric_\alpha
      + (e^{-\alpha} - 1)\tensor \Demopisobaric_\alpha
      + (2-e^{-\alpha}-e^\alpha)\tensor \Demopisobaric_\alpha = 1\tensor 1
\end{eqnarray*}
Again, we omit the very similar proof of $\Xi_q^2 = 1\tensor 1$.

We mention here that all our formul\ae\ concerning $\{\xi_v\}$,
especially the ones for $\{a_{uv}^w\}$ and for $\{\Xi_w\}$,
suggest that the proper basis
to consider is not $\{\xi_v\}$ but $\{(-1)^{\ell(v)} \xi_v\}$. 
We forebore doing so for historical reasons. One intriguing aspect of this
alternate basis is that its non-equivariant structure constants are
all nonnegative, instead of (as Brion proved \cite{BrionPositivity},
extending \cite{Buch}) alternating in sign.

\subsection{The class of the diagonal}\label{ssec:diagonal}
Let $(G/B)_\Delta$ denote the diagonal copy of $G/B$ in $(G/B)^2$,
which is invariant under the diagonal $T$-action on $(G/B)^2$. 
The corresponding Poincar\'e dual class
$\Delta_{12} \in H^*_T((G/B)^2)$ 
of this submanifold (assuming $G$ finite-dimensional) can be described
explicitly in terms of the Poincar\'e duals $S^v \in H^*_T(G/B)$ to the $X^v$.
Under the isomorphism
$$
H^*_T((G/B)^2) \cong H^*_T(G/B)\otimes_{H_T^*} H^*_T(G/B)
$$
we have from e.g. \cite[\S 3]{BrionLectures} the factorization of the diagonal
\begin{equation}
  \label{eq:factorizationH}
\Delta_{12} 
 = \sum_v S_v \tensor S^v
\quad = \sum_v S_v \tensor (\partial_v \cdot S^{1})
\end{equation}
Consider its restriction along 
$i_w\times Id:\ \{wB/B\} \times G/B \to (G/B)^2$\, :
$$
(i_w\times Id)^*(\Delta_{12}) = \sum_v S_v|_w \otimes \partial_v\cdot S^1=J_w\cdot(1\otimes S^1).
$$
While we won't directly use this suggestive calculation of the $S_v|_w$,
it will inform a similar operator-theoretic calculation of the 
$c_{uv}^w$ in the next section.
Towards that end we rephrase the equation above 
using the equivariant Euler class $e(T\, G/B)$ of the tangent bundle:
\begin{equation}
  \label{eq:J}
 \left( e(T\, G/B) \tensor 1\right)\ \Delta_{12} 
= \sum_{w\in W} (i_w \times Id)_* \big(J_w \cdot (1 \tensor S^1) \big)
\end{equation}

We sought $K$-theory versions of this, based on
\begin{equation}
  \label{eq:factorizationK}
  \Delta^K_{12}
= \sum_v \xi_v \tensor \xi^v_\circ 
= \sum_v \xi_v \tensor (\Demop_v \cdot \xi^1_\circ)
\quad =\quad
  \sum_v \xi_v^\circ \tensor \xi^v 
= \sum_v \xi_v^\circ \tensor (\Demopisobaric_v \cdot \xi^1),
\end{equation}
but didn't find them.

\junk{

The $K$-theory versions of (\ref{eq:factorizationH}) are
{\color{teal}  Recall:
$$ \Xi_w := \sum_{v}  (-1)^{\ell(v)} (\xi_v|_w)\, w \tensor \Demop_v, \qquad
\Xi^\circ_w := \sum_{v} (\xi^\circ_v|_w)w \tensor \Demopisobaric_v $$
}
whose restrictions along $i_w\times Id$ are \todo{sign issue in first one}
$$ \sum_v (\xi_v|_w) \tensor (\Demop_v \cdot \xi^1_\circ)
= (-1)^{\ell(v)} \Xi_w \cdot (\xi_1 \tensor \xi^1_\circ) 
\qquad
\sum_v (\xi_v^\circ|_w) \tensor (\Demopisobaric_v \cdot \xi^1)
= \Xi^\circ_w \cdot (\xi_1 \tensor \xi^1) 
$$
\todo{AK doesn't think the sign is fixable in the first...?
  Anyway, this is the sign that the $\Demopisobaric$ (not $\Demop$)
  chosen in the $\Xi^\circ_w$ formula is the right one, and vaguely, 
  a reason for it to be so}

\junk{

\subsection{Expanding in $(\partial_u)$ -- will we need this?}

In the next section we'll need a generalization in which the first
tensor factor lives in $H^*_T[\partial]$, not just the subring $H^*_T[W]$. 
To compute the expansion $w = \sum_u m_w^u \partial_u$ of a basis
element from the subring in the basis of the whole, 
we first apply both sides to $S_{v}$
$$  w \cdot S_v  = \sum_u m_w^u \partial_u \cdot S_{v}
      = \sum_{u\atop \ell(uv) = \ell(v)-\ell(u)} m_w^u S_{uv}
$$
then restrict to the basepoint. The LHS becomes
$$ (w\cdot S_v)|_1 = w\cdot(S_v|_{w^{-1}}) $$
whereas the RHS becomes
$$ \left(\sum_{u\atop \ell(uv) = \ell(v)-\ell(u)} m_w^u S_{uv}\right)\Bigg|_1
= \sum_{u\atop \ell(uv) = \ell(v)-\ell(u)} m_w^u (S_{uv}|_1)
= m_w^{v^{-1}} $$
Hence
$w = \sum_u \left(w\cdot (S_{u^{-1}}|_{w^{-1}})\right) \partial_u$,
and $J_w = \sum_{u,v} (S_v|_w) 
\left(w\cdot (S_{u^{-1}}|_{w^{-1}})\right) \partial_u \ \tensor\ \partial_v$.

}
endjunk}

\section{Schubert structure operators}\label{se:structureops}

\subsection{The cohomology operator $L^\alpha$}

Analogous to $J_\alpha \in H^*_T[W] \tensor \ZZ[\partial]$, 
we introduce homogeneous degree $-1$ elements $L^\alpha$ of
$H^*_T[\partial] \tensor \ZZ[\partial] \tensor \ZZ[\partial]$ 
$$
  L^\alpha 
  := (\partial_\alpha r_\alpha \tensor 1 \tensor 1) 
  + (r_\alpha \tensor \partial_\alpha \tensor 1) 
  + (r_\alpha \tensor 1 \tensor \partial_\alpha) 
  + (\alpha r_\alpha \tensor \partial_\alpha \tensor \partial_\alpha)
$$
where $r_\alpha = 1-\alpha\partial_\alpha \in H^*_T[\partial] $.
Their motivation is the following:

\begin{Proposition}\label{prop:Lc}
  Let $Q$ be a reduced word for $w$, and assume Theorem \ref{thm:main}. 
  Also assume $G$ is finite-dimensional,
  so we can define $S^u := w_0 \cdot S_{w_0 u}$. Then
  \begin{equation}
    \label{eq:Lc} \qquad\qquad
    \left(\prod_{q\in Q} L^{\alpha_q}\right) \cdot (S_1\otimes S^1 \otimes S^1) 
    = \sum_{u, v} c_{uv}^w \otimes S^u \otimes S^v
  \end{equation}
\end{Proposition}

\begin{proof}
  \begin{eqnarray*}
     \prod_{q\in Q} L^{\alpha_q} 
    &=& \sum_{R,S \subset Q}
        \prod_Q \left( \alpha_q^{[q\in R, S]} \partial_q^{\left[q\notin R, S\right]}
        r_q \right) \otimes \prod_{r\in R} \partial_r 
        \otimes \prod_{s\in S} \partial_s \\
    &=& \sum_{u,v} \left( \sum_{R,S\subset Q\text{ reduced}\atop
        \prod R = u, \ \prod S = v}
        \prod_Q \left( \alpha_q^{[q\in R, S]} \partial_q^{\left[q\notin R, S\right]}
        r_q \right) \right) \otimes \partial_u\otimes \partial_v
  \end{eqnarray*}
  The operators in the first tensor slot are the ones appearing in
  Theorem~\ref{thm:main}.
\end{proof}

Since the right side of (\ref{eq:Lc}) doesn't depend on the choice of
reduced word $Q$, this suggests that the operator $\prod_{q\in Q} L^{\alpha_q}$ 
itself might already be independent of $Q$.  We now verify this (partly by
computer\footnote{The Macaulay 2 code is available from
  the second author.}, and only in the simply- and doubly-laced cases).

\begin{Proposition}\label{prop:nilheckeaction}
  Rewrite $L^\alpha$ as $D_\alpha \tensor 1 + J_\alpha \tensor \partial_\alpha$,
  where 
  $$ D_\alpha := -\partial_\alpha \tensor 1 + r_\alpha \tensor \partial_\alpha
  \qquad\qquad\qquad
  J_\alpha := r_\alpha\tensor 1+\alpha r_\alpha\tensor\partial_\alpha $$
  recalling the latter from \S \ref{ssec:AJSBops}.
  Then the map $r_\alpha \mapsto J_\alpha$, $\partial_\alpha \mapsto -D_\alpha$,
  $\alpha \mapsto \alpha \tensor 1$ defines a homomorphism of the
  nil Hecke algebra to $H^*_T[\partial] \tensor_\ZZ \ZZ[\partial]$.
  Under the natural identification
  of $H^*_T[\partial] \tensor_\ZZ \ZZ[\partial]$ with $H^*_T[\partial] \tensor_{H^*_T} H^*_T[\partial]$,  this defines a 
  coproduct on $H^*_T[\partial]$.\footnote{This coproduct is not coassociative in the usual sense, owing to the fact that the base ring $H_T^*$ is not central.}
 
\end{Proposition}

\begin{proof}
  We need to check that the nil Hecke relations
  $$ r_\alpha^2 = 1 \qquad \alpha \partial_\alpha = 1-r_\alpha \qquad
  r_\alpha( \beta\, p) = \left(\beta - \langle \alpha,\beta\rangle \alpha\right)\, r_\alpha\, p, $$
  and the commutation and braid relations, 
  hold for $J_\alpha,D_\alpha,\alpha\tensor 1,\beta\tensor 1$.
  The squaring relation was in Theorem \ref{thm:AJSBops} and the
  second is very simple. The commutation and braid
  relations for $\{J_\alpha,J_\beta\}$ also follow from Theorem
  \ref{thm:AJSBops}. Once we check the third relation,
  \begin{eqnarray*}
    J_\alpha (\beta \tensor 1) (p_1 \tensor p_2)
    &=& (r_\alpha\tensor 1+\alpha r_\alpha\tensor\partial_\alpha)
        (\beta p_1 \tensor p_2) \\
    &=&         (r_\alpha\beta p_1 \tensor p_2) +
        (\alpha r_\alpha\beta p_1 \tensor \partial_\alpha p_2) \\
    &=&   (\left(\beta - \langle \alpha,\beta\rangle \alpha\right)\, r_\alpha
 p_1 \tensor p_2) +
     (\alpha \left(\beta - \langle \alpha,\beta\rangle \alpha\right)\, r_\alpha
 p_1 \tensor \partial_\alpha p_2) \\
    &=&   (\left(\beta - \langle \alpha,\beta\rangle \alpha\right)\tensor 1)
(r_\alpha \tensor 1 + \alpha r_\alpha \tensor  \partial_\alpha) (p_1 \tensor p_2) \\
    &=& (\left(\beta - \langle \alpha,\beta\rangle \alpha\right)\tensor 1)
        J_\alpha (p_1 \tensor p_2) 
  \end{eqnarray*}
  the commutation and braid relations for the $D_\alpha$ follow from 
  their implicit definition by $(\alpha\tensor 1)D_\alpha = J_\alpha-1$. This coproduct turns out to be the same as in  \cite[\S 2]{BerensteinRichmond}.
\end{proof}

\begin{Lemma}\label{lem:square0}
  $(L^\alpha)^2 = 0$.
\end{Lemma}
This is why in Theorem \ref{thm:main2} we need $Q$ to be a reduced word.

\begin{proof}
  Rewrite $L^\alpha$ as $D_\alpha \tensor 1 + J_\alpha \tensor \partial_\alpha$,
  where $D_\alpha = -\partial_\alpha \tensor 1 + r_\alpha \tensor \partial_\alpha$ 
  and $J_\alpha$ was defined in \S \ref{ssec:AJSBops}
  to be $r_\alpha\tensor 1+\alpha r_\alpha\tensor\partial_\alpha$.
  Now use the abstract nil Hecke relations
  $$ \partial_\alpha^2 = 0 \qquad r_\alpha \partial_\alpha
  + \partial_\alpha r_\alpha = 0 $$
  and Proposition \ref{prop:nilheckeaction} to compute
  \begin{eqnarray*}
    (L^\alpha)^2 &=& (D_\alpha \tensor 1 + J_\alpha \tensor \partial_\alpha)^2 
    = D_\alpha^2\tensor 1 + (D_\alpha J_\alpha + J_\alpha D_\alpha)\tensor
        \partial_\alpha + J_\alpha^2 \tensor (\partial_\alpha)^2 \\
    &=& 0\tensor 1 + 0 \tensor \partial_\alpha + 1 \tensor 0 \qquad =\qquad 0.
  \end{eqnarray*}
\end{proof}

\begin{proof}[Proof of Theorem \ref{thm:main2}, for $L^\alpha$.]
  The commutation relations (that $[L^\alpha,L^\beta]=0$ for $\alpha\perp \beta$) 
  are obvious. For the $A_2$ braiding, we compute
  $L^\alpha L^\beta L^\alpha$ for the simple roots in $SL_3$.

  As in the proof of Lemma \ref{lem:square0}, 
  let $L^\alpha = D_\alpha \tensor 1 + J_\alpha \tensor \partial_\alpha$,
  $L^\beta = D_\beta \tensor 1 + J_\beta \tensor \partial_\beta$.
  Then
  \begin{eqnarray*}
    L^\alpha L^\beta L^\alpha
    &=& (D_\alpha \tensor 1 + J_\alpha \tensor \partial_\alpha)
        (D_\beta \tensor 1 + J_\beta \tensor \partial_\beta)
        (D_\alpha \tensor 1 + J_\alpha \tensor \partial_\alpha)\\
    &=&         (D_\alpha \tensor 1)(D_\beta \tensor 1)(D_\alpha \tensor 1)
        +(D_\alpha \tensor 1)(D_\beta \tensor 1)(J_\alpha \tensor \partial_\alpha)
        +(D_\alpha \tensor 1)(J_\beta \tensor \partial_\beta)(D_\alpha \tensor 1)\\
    &&+(D_\alpha \tensor 1)(J_\beta \tensor \partial_\beta)(J_\alpha \tensor \partial_\alpha)
+(J_\alpha \tensor \partial_\alpha)(D_\beta \tensor 1)(D_\alpha \tensor 1)
+(J_\alpha \tensor \partial_\alpha)(D_\beta \tensor 1(J_\alpha \tensor \partial_\alpha)
\\ &&       
+(J_\alpha \tensor \partial_\alpha)(J_\beta \tensor \partial_\beta)(D_\alpha \tensor 1)
+(J_\alpha \tensor \partial_\alpha)(J_\beta \tensor \partial_\beta)(J_\alpha \tensor \partial_\alpha) \\
    &=&         (D_\alpha D_\beta D_\alpha \tensor 1) 
        +(D_\alpha D_\beta J_\alpha \tensor \partial_\alpha)
        +(D_\alpha J_\beta D_\alpha \tensor \partial_\beta) \\
    &&+(D_\alpha J_\beta J_\alpha \tensor \partial_\beta \partial_\alpha)
+(J_\alpha D_\beta D_\alpha \tensor \partial_\alpha)
+(J_\alpha D_\beta J_\alpha \tensor \partial_\alpha^2)
\\ &&       
+(J_\alpha J_\beta D_\alpha \tensor \partial_\alpha \partial_\beta)
+(J_\alpha J_\beta J_\alpha \tensor \partial_\alpha \partial_\beta \partial_\alpha) \\
    &=&     D_\alpha D_\beta D_\alpha \tensor 1
+  (D_\alpha D_\beta J_\alpha 
        +J_\alpha D_\beta D_\alpha) \tensor \partial_\alpha
        +D_\alpha J_\beta D_\alpha \tensor \partial_\beta  \\
 &&+D_\alpha J_\beta J_\alpha \tensor \partial_\beta \partial_\alpha
   +J_\alpha J_\beta D_\alpha \tensor \partial_\alpha \partial_\beta
   +J_\alpha J_\beta J_\alpha \tensor \partial_\alpha \partial_\beta \partial_\alpha 
  \end{eqnarray*}
We want this to match $L^\beta L^\alpha L^\beta$, whose 
corresponding expansion looks the same, requiring we check the equations
$$
D_\alpha D_\beta D_\alpha = D_\beta D_\alpha D_\beta \quad
  D_\alpha D_\beta J_\alpha 
  +J_\alpha D_\beta D_\alpha = D_\beta J_\alpha D_\beta \quad
  D_\alpha J_\beta J_\alpha = J_\beta J_\alpha D_\beta \quad
J_\alpha J_\beta J_\alpha = J_\beta J_\alpha J_\beta 
$$
whose analogues in the nil Hecke algebra are straightforward to check;
then we apply Proposition \ref{prop:nilheckeaction}.

The corresponding $B_2$ calculation we left to a computer.
\end{proof}

We are confident that the $L^\alpha$ satisfy the $G_2$ braid relation as well, 
but have not done the computation (having run out of memory). 

\junk{

\begin{proof}[Darij Grinberg's alternate proof.]
  After our announcement of these results \cite{FPSAC}, Darij Grinberg
  suggested to us an alternate proof technique, based on two ideas:
  \begin{itemize}
  \item Our algebra 
    $H^*_T[\partial] \tensor \ZZ[\partial] \tensor \ZZ[\partial]$ 
    acts faithfully on $H^*_T(G/B)^{\tensor 3}$, so it's enough to check
    that the actions satisfy these relations.
  \item For $p_1,p_2,p_3 \in H^*_G(G/B)$ so $r_\alpha p_i = p_i$ 
    and $\partial_i p_i = 0$, we have
    \begin{eqnarray*}
      &&L^\alpha((p_1\tensor p_2\tensor p_3)(q_1\tensor q_2\tensor q_3)) 
         = L^\alpha(p_1q_1\tensor p_2q_2\tensor p_3q_3) \\
      &=&
(-\partial_\alpha(p_1q_1)\tensor p_2q_2\tensor p_3q_3) 
+ (r_\alpha(p_1q_1)\tensor \partial_\alpha(p_2q_2)\tensor p_3q_3) \\
&&+ (r_\alpha(p_1q_1)\tensor p_2q_2\tensor \partial_\alpha(p_3q_3)) 
+ (\alpha r_\alpha(p_1q_1)\tensor  \partial_\alpha(p_2q_2)\tensor  \partial_\alpha(p_3q_3)) \\
      &=&
(-p_1 \partial_\alpha(q_1)\tensor p_2q_2\tensor p_3q_3) 
+ (p_1 r_\alpha(q_1)\tensor p_2 \partial_\alpha(q_2)\tensor p_3q_3) \\
&&+ (p_1 r_\alpha(q_1)\tensor p_2q_2\tensor p_3\partial_\alpha(q_3))     
+ (\alpha p_1 r_\alpha(q_1)\tensor  p_2 \partial_\alpha(q_2)\tensor p_3  \partial_\alpha(q_3)) \\
      &=& 
(p_1\tensor p_2\tensor p_3) \ L^\alpha((q_1\tensor q_2\tensor q_3)) 
\end{eqnarray*}
  \end{itemize}
  Consequently, it's enough to check the relations where $q_1,q_2,q_3$ each
  run through an $H^*_G(G/B)$-basis of $H^*_T(G/B)$, which are of size $|W_G|$.

  For example, to check Lemma \ref{lem:square0} take $\{1,x_1\}$ as a basis of 
  $H^*_T(GL_2/B) \iso \ZZ[x_1,x_2,y_1,y_2]/\langle x_1+x_2=y_1+y_2, x_1x_2=y_1y_2\rangle$
  over $\ZZ[y_1,y_2]$. Then there are $2^3$ checks, a typical one being
  \begin{eqnarray*}
    L^\alpha L^\alpha (x_1 \tensor x_1 \tensor 1)
    &=& 
  L^\alpha\big(
   (-\partial_\alpha x_1 \tensor x_1 \tensor 1) 
  + (r_\alpha x_1  \tensor \partial_\alpha x_1  \tensor 1) \\
  &&+ (r_\alpha x_1  \tensor x_1  \tensor \partial_\alpha 1) 
  + (\alpha r_\alpha x_1  \tensor \partial_\alpha x_1  \tensor \partial_\alpha 1)
        \big) \\
    &=& 
  L^\alpha\left(
   (-1 \tensor x_1 \tensor 1) 
  + (x_2  \tensor 1  \tensor 1) + 0 + 0
        \right) \\    
  &=& (0
  + (r_\alpha(-1) \tensor \partial_\alpha x_1 \tensor 1) + 0 + 0)
+ (0 + 0 + (-\partial_\alpha(x_2) \tensor 1 \tensor 1)   + 0) \\
    &=& (-1 \tensor 1 \tensor 1) + (1 \tensor 1 \tensor 1) \qquad = 0
  \end{eqnarray*}
  For the $G_2$ braid relation, this would involve $|W_{G_2}|^3 = 1728$ checks
  (each of, initially, $4^3$ terms). This is likely manageable but we
  didn't pursue it.
\end{proof}

endjunk}

Let $\Delta_{12} \in H^*_T\left( (G/B)^3 \right)$ denote the Poincar\'e dual of
the partial diagonal 
$\{(F_1,F_2,F_3) \in (G/B)^3\ :\ F_1 = F_2\}$,
and $\Delta_{13}$ denote that of $\{(F_1,F_2,F_3) \in (G/B)^3\ :\ F_1 = F_3\}$ 
likewise. Then $\Delta_{123} := \Delta_{12} \cap \Delta_{13}$ is the class of the full diagonal.
By two applications of Equation (\ref{eq:factorizationH}), we get
\begin{eqnarray*}
  \Delta_{123} &=& \Delta_{12} \cap \Delta_{23} 
 = \left(\sum_u (S_u \tensor S^u \tensor 1)\right)
\left(\sum_v ( S_v \tensor 1 \tensor S^v)\right) 
 = \sum_{u,v} S_u S_v \tensor S^u \tensor S^v
\\ &=& \sum_{u,v} \left( \sum_w c_{uv}^w S_w \right) \tensor S^u \tensor S^v 
 =  \sum_w (S_w \tensor 1\tensor 1) 
       \sum_{u,v} \left( c_{uv}^w \tensor S^u \tensor S^v \right)
\end{eqnarray*}
Combined with (\ref{eq:Lc}), we get 
\begin{equation}
  \label{eq:H123}
    \Delta_{123} = \sum_w (S_w \tensor 1\tensor 1)\ L^w(S_1\otimes S^1 \otimes S^1),
\end{equation}
a distinct echo of Equations (\ref{eq:factorizationH}) and (\ref{eq:J}).

{\em Question.} What is a closed form for $L^w := \prod_{q\in Q} L^{\alpha_q}$, 
analogous to that of $J^w$ in (\ref{eq:Jdef})? 

\subsection{The $K$-theory operators
  $\Lambda^\alpha,\, \Lambda_\circ^\alpha$}

The analogue of Proposition \ref{prop:Lc} for the $K$-theoretic
structure constants $\acirc_{uv}^w$ requires the operators
\begin{eqnarray*}
  \Lambda^\alpha_\circ
  &:=& \ema \ra (-\da)  \tensor 1 \tensor 1
  + \ema \ra \tensor \da \tensor 1
  + \ema \ra \tensor 1 \tensor \da
  + (1-\ema) \ra \tensor \da \tensor \da 
\junk{
  &=& E^\circ_\alpha \tensor 1
      + \Xi_\alpha^\circ \tensor \da \\
\text{where } E^\circ_\alpha 
  &:=& \ema \ra (-\da)  \tensor 1 + \ema \ra \tensor \da, \quad
       \Xi_\alpha^\circ = \ema \ra \tensor 1 + (1-\ema) \ra \tensor \da
       \text{ from \S \ref{ssec:AJSBops}}
}
\end{eqnarray*}
read off our $\{\acirc_{uv}^w\}$ formula in Theorem \ref{thm:main}, 
through considering the cases $q$ in neither of $P,R$, in just one, or in both.

To recover analogously the $a_{uv}^w$ as coefficients of an operator product,
we need to subsume the $P,R$-dependent signs into the same
four cases. 
\begin{eqnarray*}
  a_{uv}^w &=& (-1)^{\ell(u)+\ell(v)-\ell(w)} \\
           && \sum_{P,R\subset Q\atop \wt\prod P = u, \ \wt\prod R = v}
              \prod_{q\in Q} \left( (-1)^{[q \in P\cap R]} (-1)^{[q\notin P\cup R]}
              (e^{\alpha_q})^{[q\notin {P\cup R}]}  (1-e^{-\alpha_q})^{[q\in P\cap R]} 
              r_q (-\Demop_q)^{[q\notin {P}\cup {R}]}\right) \cdot 1
\end{eqnarray*}
Extracting those cases, we define
\begin{eqnarray*}
     \Lambda^\alpha
  &:=& e^{\alpha} r_\alpha \Demop_\alpha \tensor 1 \tensor 1
  + r_\alpha \tensor \Demop_\alpha \tensor 1
  + r_\alpha \tensor 1 \tensor \Demop_\alpha 
  - (1-e^{-\alpha}) r_\alpha \tensor \Demop_\alpha \tensor \Demop_\alpha 
\\
  &=& - \bda \ra \tensor 1 \tensor 1
  + \ra \tensor \bda \tensor 1
    + \ra \tensor 1 \tensor \bda
   +(\ema-1) \ra \tensor \bda \tensor \bda        
\\
\junk{
  &=& E_\alpha \tensor 1 + \Xi_\alpha \tensor \bda \\
\text{where } E_\alpha 
  &:=& - \bda \ra \tensor 1 +  \ra \tensor \bda,
       \quad
       \Xi_\alpha = \ra \tensor 1   +(\ema-1) \ra \tensor \bda
       \text{ from \S \ref{ssec:AJSBops}}
}
\end{eqnarray*}

We didn't find a zero-Hecke version of Proposition \ref{prop:nilheckeaction}
with which to study $\Lambda^\alpha, \Lambda^\alpha_\circ$.

\begin{proof}[Proof of Theorem \ref{thm:main2}, 
  for $\Lambda^\alpha, \Lambda^\alpha_\circ$.]
We check the two $K$-analogues of Lemma \ref{lem:square0}:

\begin{eqnarray*}
  (\Lambda^\alpha)^2
  &=&
  + (\bda \ra \tensor 1 \tensor 1)(\bda \ra \tensor 1 \tensor 1)
  - (\bda \ra \tensor 1 \tensor 1)(\ra \tensor \bda \tensor 1) \\  && 
  - (\bda \ra \tensor 1 \tensor 1)(\ra \tensor 1 \tensor \bda)
  - (\bda \ra \tensor 1 \tensor 1)((\ema-1) \ra \tensor \bda \tensor \bda)
\\ &&  
  - (\ra \tensor \bda \tensor 1)  (\bda \ra \tensor 1 \tensor 1)
  + (\ra \tensor \bda \tensor 1)  (\ra \tensor \bda \tensor 1)  \\ &&  
  + (\ra \tensor \bda \tensor 1)  (\ra \tensor 1 \tensor \bda)
  + (\ra \tensor \bda \tensor 1)  ((\ema-1) \ra \tensor \bda \tensor \bda)
\\ &&  %
  - (\ra \tensor 1 \tensor \bda)(\bda \ra \tensor 1 \tensor 1)
  + (\ra \tensor 1 \tensor \bda)(\ra \tensor \bda \tensor 1)  \\ &&  
  + (\ra \tensor 1 \tensor \bda)(\ra \tensor 1 \tensor \bda)
  + (\ra \tensor 1 \tensor \bda)((\ema-1) \ra \tensor \bda \tensor \bda)
\\ &&  %
  - ((\ema-1) \ra \tensor \bda \tensor \bda)(\bda \ra \tensor 1 \tensor 1)
  + ((\ema-1) \ra \tensor \bda \tensor \bda)(\ra \tensor \bda \tensor 1)\\ &&  
  + ((\ema-1) \ra \tensor \bda \tensor \bda)(\ra \tensor 1 \tensor \bda)
  + ((\ema-1) \ra \tensor \bda \tensor \bda)((\ema-1) \ra \tensor \bda \tensor \bda)
\end{eqnarray*}
\begin{eqnarray*}
    &=&
  + (\bda \ra \bda \ra \tensor 1 \tensor 1)
  - (\bda \ra \ra \tensor \bda \tensor 1) 
  - (\bda \ra \ra \tensor 1 \tensor \bda)
  - (\bda \ra (\ema-1) \ra \tensor \bda \tensor \bda)
\\ &&  
  - (\ra \bda \ra\tensor \bda \tensor 1)
  + (\ra \ra \tensor \bda \tensor 1)  
  + (\ra \ra \tensor \bda \tensor \bda) 
  + (\ra (\ema-1) \ra \tensor \bda \tensor \bda) 
\\ &&  %
  - (\ra \bda \ra \tensor 1 \tensor \bda)
  + (\ra \ra \tensor \bda \tensor \bda)
  + (\ra \ra \tensor 1 \tensor \bda)
  + (\ra (\ema-1) \ra \tensor \bda \tensor \bda)
\\ &&  %
  - ((\ema-1) \ra \bda \ra \tensor \bda \tensor \bda)
  + ((\ema-1) \ra \ra \tensor \bda \tensor \bda) \\ &&
  + ((\ema-1) \ra \ra \tensor \bda \tensor \bda)
  + ((\ema-1) \ra (\ema-1) \ra \tensor \bda \tensor \bda)
\end{eqnarray*}
which becomes, using $\ra \bda \ra = -\ema \bda$, 
\begin{eqnarray*}
    &=&
  + (- \bda \ema  \bda \tensor 1 \tensor 1)
  - (\bda \tensor \bda \tensor 1) 
  - (\bda \tensor 1 \tensor \bda)
  - (\bda (\ea-1) \tensor \bda \tensor \bda)
\\ &&  
  + (\ema \bda \tensor \bda \tensor 1)
  + (1 \tensor \bda \tensor 1)  
  + (1 \tensor \bda \tensor \bda) 
  + ((\ea-1) \tensor \bda \tensor \bda) 
\\ &&  %
  + (\ema \bda \tensor 1 \tensor \bda)
  + (1 \tensor \bda \tensor \bda)
  + (1 \tensor 1 \tensor \bda)
  + ((\ea-1) \tensor \bda \tensor \bda)
\\ &&  %
  + ((\ema-1) \ema \bda \tensor \bda \tensor \bda)
  + ((\ema-1) \tensor \bda \tensor \bda) \\ &&
  + ((\ema-1) \tensor \bda \tensor \bda)
  + ((\ema-1) (\ea-1) \tensor \bda \tensor \bda)
\end{eqnarray*}
\begin{eqnarray*}
    &=&
  - \bda \ema  \bda \tensor 1 \tensor 1
        \quad
  + (-\bda   + \ema \bda   + 1 )\tensor \bda \tensor 1
        \quad
  + (-\bda   + \ema \bda  + 1 )\tensor 1 \tensor \bda 
\\ &&
  + \big( \bda (1-\ea) 
  + 1 
  + (\ea-1) 
  + 1 
  + (\ea-1) 
  + (\ema-1) \ema \bda \\ &&
  + (\ema-1) 
  + (\ema-1) 
  + (\ema-1) (\ea-1) \big) \tensor \bda \tensor \bda \\
    &=&
  - \bda \ema  \bda \tensor 1 \tensor 1
        \quad
  + (-\bda   + \ema \bda   + 1 )\tensor \bda \tensor 1
        \quad
  + (-\bda   + \ema \bda  + 1 )\tensor 1 \tensor \bda 
\\ &&
  + \big( \bda (1-\ea) + \ea  + (\ema-1) \ema \bda + \ema
 \big) \tensor \bda \tensor \bda
\end{eqnarray*}
and then using $\bda \ema \bda = -\bda$, \ \ 
 $\bda (1-\ea) + \ea  + (\ema-1) \ema \bda + \ema = (\ema-1)\ra$,
and $-\bda + \ema \bda + 1 = \ra$ 
we reduce to $\Lambda^\alpha$.

The other $K$-analogue of Lemma \ref{lem:square0}, for
$$  \Lambda^\alpha_\circ
  := \ema \ra (-\da)  \tensor 1 \tensor 1
  + \ema \ra \tensor \da \tensor 1
  + \ema \ra \tensor 1 \tensor \da
  + (1-\ema) \ra \tensor \da \tensor \da $$
is proven by a very similar check:
\begin{eqnarray*}
&&  (\Lambda^\alpha_\circ)^2 \\
  &=& 
  + (-\ema \ra \da\tensor 1\tensor 1)(-\ema \ra \da\tensor 1\tensor 1)
  + (-\ema \ra \da\tensor 1\tensor 1)(\ema \ra\tensor \da\tensor 1)\\ &&
  + (-\ema \ra \da\tensor 1\tensor 1)(\ema \ra\tensor 1\tensor \da)
  + (-\ema \ra \da\tensor 1\tensor 1)((1-\ema) \ra\tensor \da\tensor \da)
\\ &&  %
  + (\ema \ra\tensor \da\tensor 1)(-\ema \ra \da\tensor 1\tensor 1)
  + (\ema \ra\tensor \da\tensor 1)(\ema \ra\tensor \da\tensor 1)\\ &&
  + (\ema \ra\tensor \da\tensor 1)(\ema \ra\tensor 1\tensor \da)
  + (\ema \ra\tensor \da\tensor 1)((1-\ema) \ra\tensor \da\tensor \da)
\\ &&  %
  + (\ema \ra\tensor 1\tensor \da)(-\ema \ra \da\tensor 1\tensor 1)
  + (\ema \ra\tensor 1\tensor \da)(\ema \ra\tensor \da\tensor 1)\\ &&
  + (\ema \ra\tensor 1\tensor \da)(\ema \ra\tensor 1\tensor \da)
  + (\ema \ra\tensor 1\tensor \da)((1-\ema) \ra\tensor \da\tensor \da)
\\ &&  %
  + ((1-\ema) \ra\tensor \da\tensor \da)(-\ema \ra \da\tensor 1\tensor 1)
  + ((1-\ema) \ra\tensor \da\tensor \da)(\ema \ra\tensor \da\tensor 1)\\ &&
  + ((1-\ema) \ra\tensor \da\tensor \da)(\ema \ra\tensor 1\tensor \da)
  + ((1-\ema) \ra\tensor \da\tensor \da)((1-\ema) \ra\tensor \da\tensor \da)
\end{eqnarray*}
\begin{eqnarray*}
    &=& 
  + (\ema \ra \da \ema \ra \da\tensor 1\tensor 1)
  - (\ema \ra \da\ema \ra\tensor \da\tensor 1)\\ &&
  - (\ema \ra \da\ema \ra\tensor 1\tensor \da)
  - (\ema \ra \da(1-\ema) \ra\tensor \da\tensor \da)
\\ &&  %
  - (\ema \ra \ema \ra \da\tensor \da\tensor 1)
  + (\ema \ra\ema \ra\tensor \da\tensor 1)\\ &&
  + (\ema \ra\ema \ra\tensor \da\tensor \da)
  + (\ema \ra(1-\ema) \ra\tensor \da\tensor \da)
\\ &&  %
  - (\ema \ra\ema \ra \da\tensor 1\tensor \da)
  + (\ema \ra\ema \ra\tensor \da\tensor \da)\\ &&
  + (\ema \ra\ema \ra\tensor 1\tensor \da)
  + (\ema \ra(1-\ema) \ra\tensor \da\tensor \da)
\\ &&  %
  - ((1-\ema) \ra \ema \ra \da\tensor \da\tensor \da)
  + ((1-\ema) \ra\ema \ra\tensor \da\tensor \da)\\ &&
  + ((1-\ema) \ra\ema \ra\tensor \da\tensor \da)
  + ((1-\ema) \ra(1-\ema) \ra\tensor \da\tensor \da)
\end{eqnarray*}
\begin{eqnarray*}
    &=& 
  + (\ema \ra \da \ema \ra \da\tensor 1\tensor 1)
  - (\ema \ra \da\ema \ra\tensor \da\tensor 1)\\ &&
  - (\ema \ra \da\ema \ra\tensor 1\tensor \da)
  - (\ema \ra \da(1-\ema) \ra\tensor \da\tensor \da)
\\ &&  %
  - ( \da\tensor \da\tensor 1)
  + (1\tensor \da\tensor 1)
  + (1\tensor \da\tensor \da)
  + (\ema (1-\ea)\tensor \da\tensor \da)
\\ &&  %
  - ( \da\tensor 1\tensor \da)
  + (1\tensor \da\tensor \da)
  + (1\tensor 1\tensor \da)
  + (\ema (1-\ea)\tensor \da\tensor \da)
\\ &&  %
  - ((1-\ema)  \ea \da\tensor \da\tensor \da)
  + ((1-\ema) \ea \tensor \da\tensor \da)\\ &&
  + ((1-\ema) \ea\tensor \da\tensor \da)
  + ((1-\ema) (1-\ea)\tensor \da\tensor \da)
\end{eqnarray*}
\begin{eqnarray*}
    &=& 
  + \ema \ra \da \ema \ra \da\tensor 1\tensor 1
\quad + \quad
  (- \ema \ra \da\ema \ra  -  \da  + 1)
      \tensor (\da\tensor 1 + 1\tensor \da)
\\ &&
  (- \ema \ra \da(1-\ema) \ra   + \ema    - (\ea-1) \da   + \ea )
   \tensor \da\tensor \da
\end{eqnarray*}
Using $\ra \da \ema \ra \da = - \ra \da$, 
$- \ema \ra \da(1-\ema) \ra   + \ema    - (\ea-1) \da   + \ea = (1-\ema)\ra$,
and 
$- \ema \ra \da\ema \ra  -  \da  + 1 = \ra$
we reduce the above to
$$ = - \ema \ra \da \tensor 1 \tensor 1  
+ \ra  \tensor (\da \tensor 1 + 1 \tensor \da)
+ (1-\ema)\ra \tensor \da \tensor \da \quad =\quad \Lambda^\alpha_\circ. $$
In particular, the associated graded of this equation w.r.t. the
$\langle \{1-e^\lambda\ :\ \lambda\in T^*\} \rangle$-adic filtration gives
another proof of Lemma \ref{lem:square0}.

These equations $(\Lambda^\alpha)^2 = \Lambda^\alpha,
(\Lambda^\alpha_\circ)^2 = \Lambda^\alpha_\circ$ are what imply the
last sentence of Theorem \ref{thm:main2}, allowing for $Q$ having the
correct Demazure product instead of necessarily being reduced.

\junk{
\begin{eqnarray*}
   \ra \da \ema \ra \da f 
  &=& \ra \da \ema \ra \frac{f-\ema rf}{1-\ema} \\
  &=& \ra \da \ema \frac{r f-\ea f}{1-\ea} \\
  &=& \ra \da \frac{ \ema r f- f}{1-\ea} \\
  &=& \ra \frac{1}{1-\ema} \left( \frac{ \ema r f- f}{1-\ea} - \ema r \frac{ \ema r f- f}{1-\ea} \right) \\
  &=& \ra \frac{1}{1-\ema} \left( -\ema \frac{ \ema r f- f}{1-\ema} - 
\ema \frac{ \ea  f- r f}{1-\ema} \right) \\
  &=& \ra \frac{-\ema}{1-\ema} \left(  \frac{ \ema r f- f}{1-\ema} +
 \frac{ \ea  f- r f}{1-\ema} \right) \\
  &=& \ra \frac{-\ema}{1-\ema} \left(  
f \frac{-1 + \ea}{1-\ema}
+ rf \frac{\ema - 1}{1-\ema} \right) \\
  &=& \ra \frac{-\ema}{1-\ema} \left(  f \ea - rf  \right) \\
  &=& \ra \frac{-1}{1-\ema} \left(  f - \ema rf  \right) \\
  &=& - \ra \da f
\end{eqnarray*}
}


\junk{
\begin{eqnarray*}
  (- \ema \ra \da\ema \ra  -  \da  + 1)f
  &=& -\ema \ra \frac{\ema\ra f - \ema \ra \ema \ra f}{1-\ema}
      - \frac{f - \ema \ra f}{1-\ema} + f \\
  &=& \ema \frac{ f - \ema r f}{1-\ema}
      - \frac{f - \ema r f}{1-\ema} + f \\
  &=& (\ema-1)\da f + f = (rf-f)+f = rf
\end{eqnarray*}
}

\junk{
\begin{eqnarray*}
&&    (- \ema \ra \da(1-\ema) \ra   + \ema    - (\ea-1) \da   + \ea )f \\
&=& (- \ema \ra \da(1-\ema) \ra   + \ema    - \ea (1-\ema) \da   + \ea )f \\
&=& (- \ema \ra \da(1-\ema) \ra   + \ema    - \ea (1 - \ema r)   + \ea )f \\
&=& (- \ema  d_{-\alpha} (1-\ea)    + \ema    +   r    )f \\
&=& - \ema  \frac{ (1-\ea)f - \ea \ra (1-\ea)f}{1-\ea}    + \ema f   +   r f \\
&=& - \ema  \frac{ (1-\ea)f - (\ea-1)\ra f}{1-\ea}    + \ema f   +   r f \\
&=& - \ema  (f + rf)    + \ema f   +   r f \\
&=& - \ema  ( rf)      +   r f  = (1-\ema)rf
\end{eqnarray*}
}

We checked the simply-laced braid relations
for $\Lambda^\alpha, \Lambda^\alpha_\circ$ (each involving
$>$5,000 terms) by computer, using the relations
$$
\begin{array}{cccc}
\Demopisobaric_\alpha e^\alpha = e^{-\alpha} \Demopisobaric_\alpha + e^\alpha + 1 
&
\quad
\Demopisobaric_\alpha e^\beta = e^\alpha e^\beta \Demopisobaric_\alpha - e^\alpha e^\beta
&
\quad \text{likewise with}
\\ 
\Demopisobaric_\alpha e^{-\alpha} = e^\alpha \Demopisobaric_\alpha - e^\alpha - 1
&
\quad
\Demopisobaric_\alpha e^{-\beta} = e^{-\alpha} e^{-\beta} \Demopisobaric_\alpha + e^{-\beta}    
&
\quad \alpha\leftrightarrow \beta \text{ and/or } \Demopisobaric_\alpha \leftrightarrow \Demop_\alpha
\end{array}
$$
and we likewise conjecture that the doubly- and triply-laced braid relations 
also hold. 
\end{proof}

It seems plausible that $\Lambda^w, \Lambda^w_\circ$ are adjoint in some sense
(much like $\Demop_w, w_0 \Demopisobaric_w w_0$ can be seen to be adjoint),
allowing one to prove $(\Lambda_\circ^\alpha)^2 = \Lambda_\circ^\alpha$
from $(\Lambda^\alpha)^2 = \Lambda^\alpha$.

By two applications of Equation (\ref{eq:factorizationK}) version I, we get
\begin{eqnarray*}
  \Delta_{123}^K &=& \Delta_{12}^K \cap \Delta_{23}^K 
 = \left(\sum_u (\xi^\circ_u \tensor \xi^u \tensor 1)\right)
\left(\sum_v ( \xi^\circ_v \tensor 1 \tensor \xi^v)\right) 
 = \sum_{u,v} \xi^\circ_u \xi^\circ_v \tensor \xi^u \tensor \xi^v
\\ &=& \sum_{u,v} \left( \sum_w \acirc_{uv}^w \xi^\circ_w \right) \tensor \xi^u \tensor \xi^v 
 =  \sum_w (\xi^\circ_w \tensor 1\tensor 1) 
       \sum_{u,v} \left( \acirc_{uv}^w \tensor \xi^u \tensor \xi^v \right) \\
 &=& \sum_w (\xi^\circ_w \tensor 1\tensor 1) \
     \Lambda^w_\circ (\xi_1 \tensor \xi^1 \tensor \xi^1)
\end{eqnarray*}
It is curious that in this $K$-theoretic analogue of 
Equation (\ref{eq:H123}) the input to the $\Lambda^w_\circ$ has no
ideal sheaf classes.

\section{A recursive formula 
  for $H^*_T$ structure constants}
\label{sec:recursive}

\begin{Corollary}
Fix a reflection $r_\alpha$, and
let $\overline{s}$ denote $r_\alpha s$ for $s\in W$. If $\overline{w}<w$, then
$$
c_{uv}^w
\quad=\quad
(\partial_\alpha r_\alpha) \cdot c_{uv}^{\overline{w}} 
\quad+\quad
[\overline{u} < u] c_{\overline{u}, v}^{\overline{w}} 
\quad+\quad   
[\overline{v} < v] c_{u, \overline{v}}^{\overline{w}} 
\quad+\quad  
[\overline{u} < u] [\overline{v} < v]\, \alpha\, c_{\overline{u},\overline{v}}^{\overline{w}}
$$
where $[\overline{s} < s]$ indicates $1$ if $\overline{s} < s$, and
$0$ otherwise (i.e. $\overline{s} > s$).

\end{Corollary}
We thank a referee for noting that a similar statement up to  left-right multiplication is found on \cite[page 219]{KostantKumarH}, and a $K$-theoretic version on \cite[page 565]{KostantKumarK}.

\begin{proof} 
For ease of notation, we use 
$$
\hat{c}_{uv}^w := \sum_{R, S\subset Q}
\prod_{q\in Q} \alpha_q^{[q\in R\cap S]}\partial_q^{[q\not\in R\cup S]} r_q,
$$ where the sum is over reduced words $R, S$ such that $ \prod R=u, \prod S =v$. Note that $c_{uv}^w = \hat{c}_{uv}^w \cdot 1$ and $L^w = \sum_{s,t} \hat{c}_{s,t}^ww\otimes \partial^s\otimes \partial^t$. Suppose $w=r_\alpha r_{\alpha_1}\cdots r_{\alpha_k}$ is a reduced word for $w$. Then $L^w =L^\alpha L^{\overline{w}}$, where $\overline{w} = r_\alpha w$. 
In particular
\begin{align*}
\sum_{u,v} c_{uv}^w \otimes { S^{u} \otimes S^{v} } &= L^w (S_1\otimes S^1 \otimes S^1)
= \left(L^\alpha \sum_{s, t} \hat{c}_{st}^{\overline{w}}\overline{w}\otimes \partial_s \otimes \partial_t\right) (S_1\otimes S^1 \otimes S^1)\\
& = \sum_{s, t}\big( \partial_\alpha r_\alpha \hat{c}_{st}^{\overline{w}}\overline{w} \otimes \partial_s \otimes \partial_t +r_\alpha \hat{c}_{st}^{\overline{w}}\overline{w} \otimes \partial_\alpha \partial_s \otimes \partial_t
+ r_\alpha \hat{c}_{st}^{\overline{w}}\overline{w} \otimes \partial_s \otimes  \partial_\alpha \partial_t \\
&  \hspace{.3in} +  \alpha r_\alpha \hat{c}_{st}^{\overline{w}}\overline{w} \otimes \partial_\alpha \partial_s \otimes  \partial_\alpha \partial_t\big) (S_1\otimes S^1 \otimes S^1)\\
\end{align*}
The term $\hat{c}_{uv}^w \otimes S^u\otimes S^v$ on the left is obtained as the image of $S_1\otimes S^1 \otimes S^1$ under those tensors with terms $\partial_u\otimes \partial_v$ in the second and third positions. Note that $\partial_\alpha\partial_s =\partial_{s'}$ exactly when $r_\alpha s =s'$ {\em and} $\ell(s') = \ell(s)+1$. If $r_\alpha s =s'$ but $\ell(s') \neq \ell(s)+1,$ then $\partial_\alpha\partial_s =0$.
Let $\overline{v}=r_\alpha v$ and $\overline{u}=r_\alpha u$. By matching the terms,
\begin{align*}
c_{uv}^w \otimes  { S^{u} \otimes S^{v} }  &= \big(\partial_\alpha r_\alpha \hat{c}_{uv}^{\overline{w}}\overline{w} \otimes \partial_u \otimes \partial_v 
+r_\alpha \hat{c}_{\overline{u},v}^{\overline{w}}\overline{w} \otimes \partial_\alpha \partial_{\overline{u}} \otimes \partial_v
+ r_\alpha \hat{c}_{u,\overline{v}}^{\overline{w}}\overline{w} \otimes \partial_u \otimes  \partial_\alpha \partial_{\overline{v}} \\
&  \hspace{.3in} +  \alpha \partial_\alpha \hat{c}_{\overline{u},\overline{v}}^{\overline{w}}\overline{w} \otimes \partial_\alpha \partial_{\overline{u}} \otimes  \partial_\alpha \partial_{\overline{v}} \big) (S_1\otimes S^1 \otimes S^1)\\
&= \big(\partial_\alpha r_\alpha \hat{c}_{uv}^{\overline{w}}\overline{w} \otimes \partial_u \otimes \partial_v 
+ [\overline{u} < u] \ r_\alpha \hat{c}_{\overline{u},v}^{\overline{w}}\overline{w} \otimes \partial_u \otimes \partial_v
+ [\overline{v} < v]\  r_\alpha \hat{c}_{u,\overline{v}}^{\overline{w}}\overline{w} \otimes \partial_u \otimes  \partial_v \\
& \hspace{.3in}  +  [\overline{u} < u] [\overline{v} < v]\  \alpha r_\alpha \hat{c}_{\overline{u},\overline{v}}^{\overline{w}}\overline{w} \otimes \partial_u \otimes  \partial_v \big) (S_1\otimes S^1 \otimes S^1).
\end{align*}
We evaluate the expression on the right and isolate the first tensor to obtain
\begin{align*}
c_{uv}^w &= (\partial_\alpha r_\alpha \hat{c}_{uv}^{\overline{w}}\overline{w})\cdot 1 + [\overline{u} < u]\  (r_\alpha \hat{c}_{\overline{u},v}^{\overline{w}}\overline{w} )\cdot 1+  [\overline{v} < v]\  (r_\alpha \hat{c}_{u,\overline{v}}^{\overline{w}}\overline{w} )\cdot 1 +  [\overline{u} < u] [\overline{v} < v] \ (\alpha r_\alpha \hat{c}_{\overline{u},\overline{v}}^{\overline{w}}\overline{w}) \cdot 1\\
& = (\partial_\alpha r_\alpha \hat{c}_{uv}^{\overline{w}})\cdot 1 +[\overline{u} < u] c_{\overline{u}, v}^{\overline{w}} +   [\overline{v} < v] c_{u, \overline{v}}^{\overline{w}} +  [\overline{u} < u] [\overline{v} < v]   \alpha c_{\overline{u},\overline{v}}^{\overline{w}}.
\end{align*}
\end{proof}

Fundamentally, this recurrence derives from the left action of $N(T)$ on $G/B$,
which normalizes but does not commute with the $T$-action; for this reason
one finds operators applied to the $c$ coefficients, in the first term.
In \cite[theorem 2]{noncomplex} the second author gave a similar recurrence,
but based on the $T$-equivariant {\em right} action of $W$, which thereby
only involves scaling but not operating on the $c$ coefficients.

{
We finish with an example illustrating the use of the recursive formula.
\begin{Example} We compute $c_{u,v}^w$ in the $S_3$ case, with
  $u=[312]$, $v=[132]$ and $w=w_0=[321]$ in 1-line notation. First we
  use $\overline{w}=r_1w$. Then $\overline{u}=r_1u \not< u$ and
  $\overline{v}=r_1v\not< v$. The three latter terms in the sum of the
  first recursion relationship drop out and we obtain
$$
c_{[312],[132]}^{[321]} = 
c_{uv}^w = \partial_1 r_1 \cdot c_{uv}^{\overline{w}}
 = \partial_1 r_1 \cdot c_{[312],[132]}^{[312]}
$$
We set about to compute $c_{uv}^{\overline{w}}$. Note that $r_2r_1$ is a reduced word for $\overline{w}$. There is only one subword for $u$, mainly $r_2r_1$, and one subword for $v$, mainly $r_2-$. Therefore $c_{uv}^{\overline{w}} = \alpha_2 r_2 r_1\cdot 1$ and we obtain
$$
c_{uv}^w = \partial_1 r_1 \alpha_2 r_2 r_1\cdot 1 = \partial_1 (r_1(\alpha_2))=\partial_1 (\alpha_1+\alpha_2) = 1.
$$
As a check on this result, we consider the recursion with $r_2$ 
instead of $r_1$, so $\overline{w} = r_2w=[231]$. 
Then $\overline{u}=r_2u =[213]< u$ and $\overline{v}=r_2v=1\leq v$. 
In principle all four terms are nonzero:
$$
c_{uv}^w=\partial_2 r_2 \cdot c_{uv}^{\overline{w}} +c_{\overline{u}, v}^{\overline{w}} +   c_{u, \overline{v}}^{\overline{w}} +   \alpha c_{\overline{u},\overline{v}}^{\overline{w}}.
$$
However $u \not \leq \overline{w}$, so the first and third terms
$c_{uv}^{\overline{w}}$ and $c_{u, \overline{v}}^{\overline{w}}$ vanish. 
The last term $c_{\overline{u},\overline{v}}^{\overline{w}} = c_{[213],1}^{[231]}=0$ because $S_{[213]}S_{1} = S_{[213]}$. Thus $c_{uv}^w = c_{\overline{u}, v}^{\overline{w}} = c_{[213],[132]}^{[231]}$ is the only remaining nonzero term. This smaller structure constant is easily seen to be 1, for instance by another application of same inductive formula with $r_1[231]=[132]<[231]$.  Note that $r_1[132]\not<[132]$ which forces two terms in the recursive sum to be $0$. We obtain
$$
c_{[213],[132]}^{[231]}
\quad=\quad \partial_1 r_1 \cdot c_{[213],[132]}^{[132]} \quad+\quad c_{1,[132]}^{[132]}
\quad=\quad 0\quad+\quad 1
$$
where the last two equalities follow from $[213]\not\leq [132]$ and 
 $S_1S_{[132]}=S_{[132]}$.
\end{Example}
}

\section{ Proof of Theorem \ref{th:b-formula}}
\label{se:Restrctionproof}
Here we prove the restrictions formul\ae\ from
Lemma~\ref{lemma:Trestriction} and make the inductive argument for
Theorem~\ref{th:b-formula}.  The proof of
  Lemma~\ref{lemma:Trestriction} relies on the perfect pairing
  $\langle, \rangle$ between $K_T^*(BS^Q)$ and $K^T_*(BS^Q)$, and its
  nondegenerate extension to the $K$-theory of the fixed point set
  $(BS^Q)^T$ over $frac(K_T^*)$, whose formulas are stated in
  Proposition 3, below. To prove Lemma~\ref{lemma:Trestriction} we
  introduce classes on $K_T^*((BS^Q)^T)$ with the desired restrictions,
  and use the nondegeneracy of $\langle, \rangle$ to show that they
  must in fact be the localizations of the classes $\tau^\circ_P$. The
  other statements of Lemma~\ref{lemma:Trestriction} follow
  immediately.  The inductive argument for Theorem~\ref{th:b-formula}
  requires some basic vanishing properties of the structure constants
  on $BS^Q$, established in Lemma~\ref{le:vanishing}.  Finally, we
  carry out the inductive proof with an exhaustive case check.   We
give the proofs only in $K$-theory, as the cohomology version follows
by taking associated graded w.r.t.  the
$\langle \{1-e^\lambda\ :\ \lambda\in T^*\} \rangle$-adic filtration.

\begin{Proposition}\label{prop:woodshole}
  Let $\gamma \in K_T^*(BS^Q)$, and $R\subset Q$. Then the perfect
  $K_T^*$-valued pairing $\langle,\rangle$ between $K_T^*(BS^Q)$ and
  $K^T_*(BS^Q)$ can be computed at $[\calO_{BS^R}], [I_{BS^R}]
  \in K_*^T(BS^Q)$ as
\begin{align*}
   \langle \gamma, [\calO_{BS^R}] \rangle  &=\sum_{J \subset R} \frac{\gamma|_J} 
      {\prod\limits_{i\in R} (1-e^{(\prod_{j\in J, j\leq i} r_j)\alpha_i})}\\
  \langle \gamma, [I_{BS^R}] \rangle &=      \sum_{J \subset R} \frac{\gamma|_J         }
        {\prod\limits_{i\in R}  (e^{-(\prod_{j\in J, j\leq i} r_j)\alpha_i})^{[i\notin J]}
        (1-e^{(\prod_{j\in J, j\leq i} r_j)\alpha_i})}  
 \end{align*}
  We can extend the pairing to a nondegenerate $frac(K_T^*)$-valued pairing
  of $K_T^*((BS^Q)^T)$ and $K^T_*(BS^Q)$, still called $\langle,\rangle$.
\end{Proposition}

\begin{proof}
  Using the push-pull formula applied to the inclusion $BS^R \into BS^Q$, 
  we can compute the first pairing by applying Atiyah-Bott's Woods Hole 
  formula for $K_T^*$-integration to the submanifold $BS^R$.
  That in turn requires determining the weights in the tangent space at $J$ of $BS^R$,
computed in \cite[Equation (23)]{ElekLu}.

  For the second formula, 
  we use $[I_{BS^R}] = \sum_{S\subset R} (-1)^{R\setminus S} [\calO_{BS^S}]$, so
  \begin{eqnarray*}
    \langle \gamma, [I_{BS^R}] \rangle 
    &=& \sum_{S\subset R} (-1)^{R\setminus S} \langle \gamma, [\calO_{BS^S}] \rangle 
    = \sum_{S\subset R} (-1)^{R\setminus S} 
        \sum_{J \subset S} \frac{\gamma|_J} 
        {\prod\limits_{i\in S} (1-e^{(\prod_{j\in J, j\leq i} r_j)\alpha_i})} \\
    &=& \sum_{J \subset R} \gamma|_J
        \sum_{S\subset R,\, S\supseteq J} (-1)^{R\setminus S} 
        \frac{1}
        {\prod\limits_{i\in S} (1-e^{(\prod_{j\in J, j\leq i} r_j)\alpha_i})} \\
    &=& \sum_{J \subset R} \gamma|_J
        \sum_{S\subset R,\, S\supseteq J} (-1)^{R\setminus S} 
        \frac{\prod\limits_{i\in R\setminus S} (1-e^{(\prod_{j\in J,j\leq i} r_j)\alpha_i})}
        {\prod\limits_{i\in R} (1-e^{(\prod_{j\in J, j\leq i} r_j)\alpha_i})} 
        \qquad \text{let $S^c = R\setminus S$} \\
    &=& \sum_{J \subset R} \frac{\gamma|_J}
        {\prod\limits_{i\in R} (1-e^{(\prod_{j\in J, j\leq i} r_j)\alpha_i})} 
        \sum_{S^c\subset R\setminus J} 
        {\prod\limits_{i\in S^c} (e^{(\prod_{j\in J, j\leq i} r_j)\alpha_i}-1)} \\
    &=& \sum_{J \subset R} \frac{\gamma|_J}
        {\prod\limits_{i\in R} (1-e^{(\prod_{j\in J, j\leq i} r_j)\alpha_i})} 
        \prod_{i \in R\setminus J} (1 + (e^{(\prod_{j\in J, j\leq i} r_j)\alpha_i}-1)) \\
    &=& \sum_{J \subset R} \frac{\gamma|_J
        \prod_{i \in R\setminus J} e^{(\prod_{j\in J, j\leq i} r_j)\alpha_i} }
        {\prod\limits_{i\in R} (1-e^{(\prod_{j\in J, j\leq i} r_j)\alpha_i})}         
    = \sum_{J \subset R} \frac{\gamma|_J         }
        {\prod\limits_{i\in R}  (e^{-(\prod_{j\in J, j\leq i} r_j)\alpha_i})^{[i\notin J]}
        (1-e^{(\prod_{j\in J, j\leq i} r_j)\alpha_i})}         
  \end{eqnarray*}

  The Woods Hole formula $\sum_{J \subset R} \left(\gamma|_J \, \big{/}\, 
      {\prod\limits_{i\in R} (1-e^{(\prod_{j\in J, j\leq i} r_j)\alpha_i})}\right)$
  is visibly $frac(K_T^*)$-valued, for any system
  $(\gamma|_J) \in K_T^*((BS^Q)^T)$ of point restrictions. Since both
  $K_T^*(BS^Q)$ and $K^T_*(BS^Q)$ are free $K_T^*$-modules, the nondegenerate
  pairing extends to their rationalizations (extension of scalars
  using $frac(K_T^*)\tensor_{K_T^*}$), and restricts to nondegenerate pairings
  between any full rank $K_T^*$-sublattice of $frac(K_T^*)\tensor_{K_T^*} K_*^T(BS^Q)$
  against any full rank $K_T^*$-sublattice of \break 
  $frac(K_T^*)\tensor_{K_T^*} K_T^*(BS^Q)$.
  We apply this to the sublattices $K_*^T(BS^Q)$ and $K_T^*((BS^Q)^T)$, 
  respectively, seeing the latter as a submodule of the vector space
  $frac(K_T^*)\tensor_{K_T^*} K_T^*(BS^Q)$ since the map
  $K_T^*(BS^Q) \into K_T^*((BS^Q)^T)$ becomes an isomorphism upon
  rationalization.
\end{proof}

\begin{proof}[Proof of Lemma~\ref{lemma:Trestriction}.] 

Define $\gamma_R^\circ$ 
  not in $K_T^*(BS^Q)$, but only as an element of $K_T^*$ of the fixed point set 
  $(BS^Q)^T \iso \{S\ :\ S\subset Q\}$:
  $$
  \gamma_R^\circ|_S := \sum_{J:\ R\subset J \subset S} 
  (-1)^{|J\setminus R|} \prod_{t\in J} \left(1-e^{\left(\prod_{j\in S, j\leq t} r_j\right)\alpha_t}\right)
  \quad\in K_T^*(S) $$
  We argue that
  $\gamma_R^\circ$ has the same point restrictions as the desired dual
  basis element $\tau_R^\circ$, by showing that $\gamma_R^\circ,\tau_R^\circ$ 
  have the same pairings with $\{[\calO_{BS^S}]\}$, and then invoking
  nondegeneracy of the pairing.

By definition of dual basis, $\langle\tau_R^\circ, [\calO_{BS^P}]\rangle = [R=P]$.
On the other hand, if we pair $\gamma_R^\circ$ 
against $[\calO_{BS^P}]$ using $\langle, \rangle$, the expression simplifies:
\begin{align}
\sum_S \frac{\gamma^\circ_R|_S }{\prod\limits_{i\in P} (1-e^{(\prod_{j\in S, j\leq i} r_j)\alpha_i})} 
& =
 \sum\limits_{S:\ S\subset P}\frac{ 
\sum\limits_{J:\ R\subset J \subset S} (-1)^{|J\setminus R|} 
\prod\limits_{t\in J} 
\left(1-e^{\left(\prod_{j\in S, j\leq t} r_j\right)\alpha_t}\right)}
{\prod\limits_{i\in P} (1-e^{(\prod_{j\in S, j\leq i} r_j)\alpha_i})}\nonumber \\
&  =\sum_{J:\ R\subset J \subset P} 
(-1)^{|J\setminus R|} \sum_{S:\  J \subset S\subset P}
\frac{1}
{\prod\limits_{i\in P\setminus J} (1-e^{(\prod_{j\in S, j\leq i} r_j)\alpha_i})}.\label{rationalsum1}
 \end{align}  
We claim that, for each $J$, the sum $\sum\limits_{S:\  J \subset S\subset P}
\frac{1}
{\prod\limits_{i\in P\setminus J} (1-e^{(\prod_{j\in S, j\leq i} r_j)\alpha_i})} =1$.   When  $|P|=0$, the sum is over the single choice $J=S=P=\emptyset$ of the empty product, returning 1. Suppose 
$P$ is nonempty and that $a$ is the first letter of $P$. 
 
 If $a\in J$ (as its first term), then all $S$ in the summing set contain $a$ as well. Denote $P_0 = P\setminus \{a\} $ and similarly for $J_0$. Since  $a\not\in P\setminus J$,  the sum 
 simplifies to 
 $$
  \sum\limits_{S:\ J_0\subset S\setminus\{a\} \subset P_0 }\frac{1}{\prod\limits_{i\in P\setminus J}   (1-e^{(\prod_{j\in S\setminus \{a\} , j\leq i} r_j)\alpha_i})} = 
  \sum_{S:\ J_0 \subset S \subset P_0 }\frac{1}{\prod\limits_{i\in P_0\setminus J_0}  (1-e^{(\prod_{j\in S , j\leq i} r_j)\alpha_i})} = 1,
 $$
 where the last equality follows by induction.

On the other hand, if $a\not\in J$, the sum
\begin{align*}
&=\sum_{S:\ a\in S \atop J\subset S\subset P}  \frac{1}{\prod\limits_{i\in P\setminus J} (1-e^{(\prod_{j\in S, j\leq i} r_j)\alpha_i})} + \sum_{S:\ J\subset S\subset P\atop a\not\in S}  \frac{1}{\prod\limits_{i\in P\setminus J} (1-e^{(\prod_{j\in S, j\leq i} r_j)\alpha_i})} \\
&=\sum_{S:\  a\in S \atop J\subset S\subset P}  \frac{1}{(1-e^{-\alpha_a})\prod\limits_{i\in P_0\setminus J} (1-e^{r_a(\prod_{j\in S\setminus\{a\}, j\leq i} r_j)\alpha_i})}  +  \sum_{S:\  a\not\in S \atop J\subset S\subset P}  \frac{1}{(1-e^{\alpha_a})\prod\limits_{i\in P_0\setminus J} (1-e^{(\prod_{j\in S, j\leq i} r_j)\alpha_i})} \\
&=\frac{1}{(1-e^{-\alpha_a})}r_a\sum_{S:\  a\in S \atop J\subset S\setminus \{a\}\subset P_0}  \frac{1}{\prod\limits_{i\in P_0\setminus J} (1-e^{(\prod_{j\in S\setminus \{a\}, j\leq i} r_j)\alpha_i})}   + \frac{1}{(1-e^{\alpha_a})} \sum_{S:\  a\not\in S}  \frac{1}{\prod\limits_{i\in P_0\setminus J} (1-e^{(\prod_{j\in S, j\leq i} r_j)\alpha_i})} \\
&= \frac{1}{(1-e^{-\alpha_a})}r_a(1) + \frac{1}{(1-e^{\alpha_a})}(1)=1\\
\end{align*}
where the last two equalities follow from induction and a simple calculation. By using this in $\eqref{rationalsum1}$ we obtain
$$
\langle \gamma_R^\circ, [\calO_{BS^P}]\rangle = \sum_{J:\ R\subset J\subset P} (-1)^{|J\setminus R|} =  (1-1)^{|P\setminus R|} = [R = P].
$$

 Finally, we use Proposition \ref{prop:woodshole}'s nondegeneracy of the
  pairing to infer $\tau_R^\circ = \gamma_R^\circ$.

  The $\tau_R|_S$ result is then derived 
  using
  $\tau_R = \sum\limits_{P: P\supset R}  \tau^\circ_P$ (Equation \eqref{eq:tautaucirc}):
  \begin{align*}
\left(\prod_{m\in S} (1-e^{-\alpha_m})^{[m\in R]} r_m\right)\cdot 1   &= \left(\prod_{m\in S} (1-e^{-\alpha_m}+e^{-\alpha_m})^{[m\not\in R]}(1-e^{-\alpha_m})^{[m\in R]} r_m\right)\cdot 1  \\
&=\sum\limits_{P: R\subset P \subset S} \left( \prod_{m\in S} (e^{-\alpha_m})^{[m\not\in P]} (1-e^{-\alpha_m})^{[m\in P]} r_m\right)\cdot 1  \\
& =  \sum\limits_{P: P\supset R}  \tau^\circ_P|_S = \tau_R|_S,
  \end{align*}
  as stated in Lemma~\ref{lemma:Trestriction}.
The  $T_R|_S$ result is obtained by limiting from $K_T^*$ to $H^*_T$ as usual.
\end{proof}

\begin{table}[h] 
  \caption{\label{table:listrestrictions}  Some relations between point restrictions in $K_T^*(BS^Q)$, easily proven 
    from Lemma \ref{lemma:Trestriction}. }
\begin{center}
\begin{tabular}{|c|c|c|}
\hline &&\\
&$U$ begins with $r_\alpha$ &$U$ doesn't begin with $r_\alpha$ \\
 &$U=r_\alpha U_0$ & \\
&&\\ \hline &&\\

$J = r_\alpha J_0$ & $\tau^\circ_U|_J = (1-e^{-\alpha })r_\alpha  \left( \tau_{U_0}^\circ |_{J_0}\right)$  &  $\tau^\circ_U|_J = e^{-\alpha } r_\alpha \left( \tau^\circ_U|_{J_0}\right)$\\
&&\\ \hline && \\
$J = r_\alpha J_0$ & $\tau_U|_J = (1-e^{-\alpha })r_\alpha \left(\tau_{U_0}|_{J_0}\right)$ &  $\tau_{U}|_J =r_\alpha  \left(\tau_{U}|_{J_0}\right)$\\
&&\\ \hline 
\end{tabular}
\end{center}
\label{default}
\end{table}

In order to prove the main theorem, we need a basic restriction property:

\begin{Lemma}\label{le:vanishing} 
  The structure constants $b_{RS}^J, d_{RS}^J, \dcirc_{RS}^J$ vanish
  unless $J \supseteq R,S$.
\end{Lemma}

\begin{proof}
  $$
  d_{RS}^J = \langle \tau_R \tau_S, [BS^J] \rangle
  = \sum_{L \subset J} \frac{(\tau_R)|_L (\tau_S)|_L}{\prod\limits_{i\in J} (1-e^{(\prod_{j\in L, j\leq i} r_j)\alpha_i})}
  = \sum_{L \subset J,\ L \supseteq R,S} \frac{(\tau_R)|_L (\tau_S)|_L}{\prod\limits_{i\in J} (1-e^{(\prod_{j\in L, j\leq i} r_j)\alpha_i}))} 
  $$
  where the last equality follows from vanishing properties of $\tau_R,\tau_S$.
  This sum is empty unless $J \supseteq R,S$. The proofs for
  the other two families of structure constants are exactly the same.
\end{proof}

\begin{proof}[Proof of Theorem~\ref{th:b-formula}] 
  We prove the theorem for $\{d_{RS}^J\}$ and $\{\dcirc_{RS}^J\}$ inductively, 
  where formulas for $\{d^Q_{RS}\}$ and $\{\dcirc^Q_{RS}\}$ for $Q\subsetneq J$ form the inductive assumption. 
  The result for the $H_T$-coefficients $\{b^Q_{RS}\}$
  follows by taking the associated graded to get the statement in cohomology. There are four cases indicated below in bold before their proofs. The first case will serve as the base for the induction and the others as the inductive step. 

\noindent {\bf \underline{Case 1:} $J=\emptyset$.}

As $J=\emptyset$, $R,S\subset J$ implies $R,S=\emptyset$. An easy calculation from the restriction lemma shows that $\tau_\emptyset |_\emptyset = \tau_\emptyset^\circ |_\emptyset = 1$. Thus
 $$
 \tau_R |_J \tau_S |_J = 1 = d_{RS}^J \tau_J |_J =  d_{RS}^J \quad \mbox{ and }\quad
 \tau^\circ_R |_J \tau^\circ_S |_J = 1 = \dcirc_{RS}^J \tau^\circ_J |_J =  \dcirc_{RS}^J.
 $$
 Meanwhile, the right hand sides of the equations in Theorem~\ref{th:b-formula}
 are empty products, verifying the theorem in this case.

Having dealt with the case  $J=\emptyset$, we assume henceforth that $J$ has at least one element. 
By way of induction assume that Theorem~\ref{th:b-formula} holds for all $Q\subsetneq J$.  
 Using the vanishing properties (Lemma~\ref{le:vanishing}) we obtain
 $$
\tau_S|_J \tau_R|_J = d_{RS}^J \tau_J|_J + \sum_{Q\ :\ Q\subsetneq J} d_{RS}^Q \tau_Q|_J \qquad\qquad
\tau^\circ_S|_J \tau^\circ_R|_J = \dcirc_{RS}^J \tau^\circ_J|_J + \sum_{Q\ :\  Q\subsetneq J} \dcirc_{RS}^Q \tau^\circ_Q|_J 
$$
since  $J\subset Q$, implies either $Q=J$ or $\tau_Q|_J=\tau^\circ_Q|_J=0$.
Solving for $d_{RS}^J$ and $\dcirc_{RS}^J$ (with a division justified by
the last statement in Lemma \ref{lemma:Trestriction})
\begin{align}
d_{RS}^J &= \frac{\tau_S|_J \tau_R|_J - \sum_{Q: \ Q\subsetneq J} d_{RS}^Q \tau_Q|_J}{\tau_J|_J}, \quad\mbox{ and}\label{eq:inductivestep}\\
\dcirc_{RS}^J &= \frac{\tau^\circ_S|_J \tau^\circ_R|_J - \sum_{Q:\ Q\subsetneq J} \dcirc_{RS}^Q \tau^\circ_Q|_J}{\tau^\circ_J|_J}.\label{eq:Kinductivestep}
\end{align}

\noindent {\bf \underline{Case 2:} Neither $R$ nor $S$ contains the first letter of $J$.}

\begin{Lemma}\label{lemma:simplification} 
  Let $J_0$ be defined by $J=r_\alpha J_0$, where $r_\alpha $ is the
  simple reflection in the first position of $J$. Suppose also that
  neither $R$ nor $S$ contains that first letter of $J$. Then
\begin{align*}
  \sum_{Q\ : \ Q\subsetneq J} d_{RS}^Q \tau_Q|_J 
  &= \phantom{e^{-\alpha }\bigg(}
    \sum_{Q:\ Q\subsetneq J_0} 
    \phantom{e^{-\alpha }} r_\alpha ({d}_{RS}^{Q}\tau_Q|_{J_0})  
    +  d_{RS}^{J_0} r_\alpha (\tau_{J_0}|_{J_0}) \\
  \sum_{Q: \ Q\subsetneq J} \dcirc_{RS}^Q \tau^\circ_Q|_J 
  &= e^{-\alpha }\left( \sum_{Q:\ Q\subsetneq J_0}  
    e^{-\alpha } r_\alpha (\dcirc_{RS}^Q \tau^\circ_Q|_{J_0}) 
    + \dcirc_{RS}^{J_0} r_\alpha ( \tau^\circ_{J_0}|_{J_0})\right),
\end{align*}
where the sum is over those $Q$ that contain both $R$ and $S$, since otherwise $d_{RS}^Q = \dcirc_{RS}^Q=0$.
\junk{\color{blue}
\begin{align*}
\sum_{Q: \ |Q|\leq k, R,S\subset Q} d_{RS}^Q \tau_Q|_J &= \sum_{ |Q|\leq k-1} r_m({d}_{RS}^{Q}\tau_Q|_{J_0})  +  d_{RS}^{J_0} r_m(\tau_{J_0}|_{J_0}) \\
\sum_{Q: \ |Q|\leq k, R,S\subset Q} \dcirc_{RS}^Q \tau^\circ_Q|_J &= e^{-\alpha_m}\left( \sum_{|Q|\leq k-1}  e^{-\alpha_m} r_m(\dcirc_{RS}^Q \tau^\circ_Q|_{J_0}) + \dcirc_{RS}^{J_0} r_m( \tau^\circ_{J_0}|_{J_0})\right) \\
\end{align*}
}
\end{Lemma}
\begin{proof}[Proof of Lemma \ref{lemma:simplification}]

If $Q$ begins with $r_\alpha $, the inductive assumptions from Theorem~\ref{th:b-formula} stated with $Q_0$ defined by $Q = r_\alpha Q_0$ are as follows:
\begin{align*}
d_{RS}^Q &= e^{\alpha } r_\alpha  (-\Demop_\alpha )\prod_{q\in Q_0} 
(e^{\alpha_q})^{[q\not\in R\cup S]} (1-e^{-\alpha_q})^{[q\in R\cap S]} r_q (-\Demop_\alpha )^{[q\not\in R\cup S]}\cdot 1\\
&= e^{\alpha } r_\alpha  (-\Demop_\alpha ) \left( d_{RS}^{Q_0}\right)\cdot 1= \Demop_\alpha  r_\alpha   \left( d_{RS}^{Q_0}\right)\cdot 1
\end{align*}
and, similarly,
\begin{align*}
\dcirc_{RS}^Q &= e^{-\alpha } r_\alpha  (-\Demopisobaric_\alpha )\prod_{q\in Q_0} (e^{-\alpha_q})^{[q\not\in R\cap S]} (1-e^{-\alpha_q})^{[q\in R\cap S]} r_q (-\Demopisobaric_q)^{[q\not\in R\cup S]}\cdot 1\\
& =  e^{-\alpha } r_\alpha  (-\Demopisobaric_\alpha )\left( \dcirc_{RS}^{Q_0} \right)\cdot 1 =  -e^{-\alpha }\Demopisobaric_\alpha  \left(\dcirc_{RS}^Q\right)\cdot 1.
\end{align*}
If $Q\subset J$ does not begin with $r_\alpha $, then $\tau_Q |_J \neq 0$ (or $\tau_Q^\circ |_J \neq 0$) implies $Q\subset J_0$ where $J_0$ is defined by $J=r_\alpha J_0$.
Thus we break the sum into those $Q$ containing the first letter $r_\alpha$ and those not.
\begin{align*}
\sum_{Q: \ Q\subsetneq J} d_{RS}^Q \tau_Q|_J  & = \sum_{ Q\subsetneq J, r_\alpha \in Q}  d_{RS}^Q \tau_Q|_J + \sum_{Q \subsetneq J, r_\alpha \not\in Q} d_{RS}^Q \tau_Q|_J \\
& = \sum_{ Q\subsetneq J, r_\alpha \in Q}  d_{RS}^Q \tau_Q|_J + \sum_{Q \subset J_0}
d_{RS}^Q r_\alpha \left(\tau_Q|_{J_0} \right)\quad\mbox{using Table~\ref{table:listrestrictions}}\\
& = \sum_{Q\subsetneq J, Q=r_\alpha Q_0,
\atop \mbox{\tiny for some }Q_0}  d_{RS}^Q \tau_Q|_J + \sum_{Q\subsetneq J_0} d_{RS}^Q r_\alpha \left(\tau_Q|_{J_0}\right)  + d_{RS}^{J_0} r_\alpha \left(\tau_{J_0}|_{J_0}\right)\\ 
& = \sum_{Q\subsetneq J, Q=r_\alpha Q_0,
\atop \mbox{\tiny for some }Q_0}  \left[\Demop_\alpha  r_\alpha  \left(d_{RS}^{Q_0}\right)\cdot 1\right]\left[\tau_Q|_J\right] + \sum_{Q\subsetneq J_0} d_{RS}^Qr_\alpha \left(\tau_Q|_{J_0}\right)  + d_{RS}^{J_0} r_\alpha \left( \tau_{J_0}|_{J_0}\right),\\
\end{align*}

\noindent 
by the inductive hypothesis, as neither $R$ nor $S$ begins with $r_\alpha $. 
We expand the operators:
\begin{align*}
& = \sum_{Q\subsetneq J, Q=r_\alpha Q_0,
\atop \mbox{\tiny for some }Q_0}  \left[\frac{r_\alpha  \left(d_{RS}^{Q_0}\right)\cdot 1 - d_{RS}^{Q_0}\cdot 1}{1-e^{-\alpha }}\right]\left[(1-e^{-\alpha }) r_\alpha  \left(\tau_{Q_0}|_{J_0}\right)\right] + \sum_{Q \subsetneq J_0 } d_{RS}^Q r_\alpha \left(\tau_Q|_{J_0}\right) + d_{RS}^{J_0} r_\alpha \left(\tau_{J_0}|_{J_0}\right)\\
&= \sum_{Q\subsetneq J, \atop Q=r_\alpha Q_0}  \left[r_\alpha  \left(d_{RS}^{Q_0}\right)\cdot 1\right]\left[r_\alpha  \left(\tau_{Q_0}|_{J_0}\right)\right] 
-\sum_{Q\subsetneq J, \atop Q=r_\alpha Q_0}  d_{RS}^{Q_0} r_\alpha  \left(\tau_{Q_0}|_{J_0}\right) + 
\sum_{Q\subsetneq J_0} d_{RS}^{Q} r_\alpha \left(\tau_{Q}|_{J_0}\right)  + d_{RS}^{J_0} r_\alpha \left(\tau_{J_0}|_{J_0}\right)\\ 
& =\phantom{\ } \sum_{R\subsetneq J_0} \phantom{Q}
\left[r_\alpha  \left(d_{RS}^{R}\right)\cdot 1\right] \left[r_\alpha  \left(\tau_{R}|_{J_0}\right)\right] 
-\phantom{\ }
\sum_{R\subsetneq J_0}  d_{RS}^{R} r_\alpha  \left(\tau_{R}|_{J_0} \right)+
\sum_{Q\subsetneq J_0} d_{RS}^Q r_\alpha \left( \tau_{Q}|_{J_0} \right) + d_{RS}^{J_0} r_\alpha \left(\tau_{J_0}|_{J_0}\right)\\
&= \phantom{\ }\sum_{ R\subset J_0} \phantom{Q} r_\alpha  \left(d_{RS}^{R}\tau_R|_{J_0} \right) 
\phantom{\left[r_\alpha  \tau_{Q}|_{J_0}\right]  -\sum_{Q\subsetneq J_0}  d_{RS}^{Q} r_\alpha  \tau_{Q}|_{J_0} +
\sum_{Q\subsetneq J_0, \atop Q\neq r_\alpha Q_0} d_{RS}^Q r_\alpha \tau_{Q}|_{J_0}  }
  + d_{RS}^{J_0} r_\alpha \left( \tau_{J_0}|_{J_0}\right), 
\end{align*}
\junk{\color{blue}
\begin{align*}
& = \sum_{|Q|\leq k, Q=r_mQ_0,
\atop \mbox{\tiny for some }Q_0}  \left[\frac{r_m \left(d_{RS}^{Q_0}\right)\cdot 1 - d_{RS}^{Q_0}\cdot 1}{1-e^{-\alpha }}\right]\left[(1-e^{-\alpha }) r_m \tau_{Q_0}|_{J_0}\right] + \sum_{|Q|\leq k-1 \atop Q\neq r_mQ_0,
} d_{RS}^Q \tau_Q|_J + d_{RS}^{J_0} \tau_{J_0}|J\\
&= \sum_{|Q|\leq k, \atop Q=r_mQ_0}  \left[r_m \left(d_{RS}^{Q_0}\right)\cdot 1\right]\left[r_m \tau_{Q_0}|_{J_0}\right] 
-\sum_{|Q|\leq k, \atop Q=r_mQ_0}  d_{RS}^{Q_0} r_m \tau_{Q_0}|_{J_0} + 
\sum_{|Q|\leq k-1, \atop Q\neq r_mQ_0} d_{RS}^Q r_m\tau_{Q}|_{J_0}  + d_{RS}^{J_0} r_m\tau_{J_0}|_{J_0}\\. 
& = \sum_{|Q|\leq k-1}  \left[r_m \left(d_{RS}^{Q}\right)\cdot 1\right]\left[r_m \tau_{Q}|_{J_0}\right] 
-\sum_{|Q|\leq k-1}  d_{RS}^{Q} r_m \tau_{Q}|_{J_0} + 
\sum_{|Q|\leq k-1, \atop Q\neq r_mQ_0} d_{RS}^Q r_m\tau_{Q}|_{J_0}  + d_{RS}^{J_0} r_m\tau_{J_0}|_{J_0}\\
&= \sum_{ |Q|\leq k-1} r_m \left(d_{RS}^{Q}\tau_Q|_{J_0} \right) 
\phantom{\left[r_m \tau_{Q}|_{J_0}\right]  -\sum_{|Q|\leq k-1}  d_{RS}^{Q} r_m \tau_{Q}|_{J_0} +
\sum_{|Q|\leq k-1, \atop Q\neq r_mQ_0} d_{RS}^Q r_m\tau_{Q}|_{J_0}  }
  + d_{RS}^{J_0} r_m\tau_{J_0}|_{J_0},
\end{align*}
}
\junk{where the change of index from the second to third line is possible because 
$Q\subset J_0$ implies $Q$ does not begin with $r_m$. {\color{blue} And this is where we use that $Q$ is reduced.}}

The proof is similar in the $\{\tau^\circ_Q\}$ basis. Using restrictions in Table \ref{table:listrestrictions} and  vanishing properties, 
\junk{{\color{blue}
\begin{align*}
\sum_{Q: \ |Q|\leq k, \atop R,S\subset Q} \dcirc_{RS}^Q \tau^\circ_Q|_J &= \sum_{|Q|\leq k, Q=r_mQ_0,
\atop \mbox{\tiny for some }Q_0}  \dcirc_{RS}^Q \tau^\circ_Q|_J + \sum_{|Q|\leq k, Q\neq r_mQ_0,
\atop \mbox{\tiny for any }Q_0} \dcirc_{RS}^Q \tau^\circ_Q|_J \\
& = \sum_{Q\subsetneq J, Q=r_mQ_0,
\atop \mbox{\tiny for some }Q_0}  \dcirc_{RS}^Q \tau^\circ_Q|_J + \sum_{|Q|\leq k-1, Q\neq r_mQ_0,
\atop \mbox{\tiny for any }Q_0}  \dcirc_{RS}^Q \tau^\circ_Q|_J  + \dcirc_{RS}^{J_0} \tau^\circ_{J_0}|_J\\
\end{align*}
}
}
\begin{align*}
\sum_{Q: \ Q\subsetneq J \atop R,S\subset Q} \dcirc_{RS}^Q \tau^\circ_Q|_J &= \sum_{Q\subsetneq J, r_\alpha \in Q}  \dcirc_{RS}^Q \tau^\circ_Q|_J + \sum_{Q\subsetneq J, r_\alpha \not\in Q} \dcirc_{RS}^Q \tau^\circ_Q|_J  \\
& = \sum_{Q\subsetneq J, Q=r_\alpha Q_0,
\atop \mbox{\tiny for some }Q_0}  \dcirc_{RS}^{r_\alpha Q_0}  (1-e^{-\alpha})  r_\alpha (\tau^\circ_{Q_0}|_{J_0})+ \sum_{Q\subsetneq J_0}  \dcirc_{RS}^Q e^{-\alpha}r_\alpha (\tau^\circ_Q|_{J_0})  + \dcirc_{RS}^{J_0} r_\alpha (\tau^\circ_{J_0}|_{J_0})\\
\end{align*}
where the equality on the second line follows from the restriction properties, as well as that $r_\alpha $ is the first letter of $Q$, if $Q$ contains it.
We realign the index sets, noting that the index set $\{Q:\ Q\subsetneq J, Q=r_\alpha  Q_0\}$ of the first sum equals
$\{r_\alpha Q_0: \ Q_0\subsetneq J_0\}$  which is in $1$-$1$ correspondence with $\{Q_0: \ Q_0\subsetneq J_0\}$, which in turn is the same as the index set for the second sum. The latter set can be reindexed as $\{R: R\subsetneq J_0\}$. We thus obtain
$$
\sum_{Q: \ Q\subsetneq J} \dcirc_{RS}^Q \tau^\circ_Q|_J 
= \sum_{R\subsetneq J_0} \left( \dcirc_{RS}^{r_\alpha R}  (1-e^{-\alpha})  r_\alpha (\tau^\circ_{R}|_{J_0})+  \dcirc_{RS}^R e^{-\alpha_q}r_\alpha \left(\tau^\circ_R|_{J_0}\right)\right)  + \dcirc_{RS}^{J_0} r_\alpha \tau^\circ_{J_0}|_{J_0}
$$


By the inductive hypothesis that $\dcirc_{RS}^{r_\alpha R} = -e^{-\alpha}\Demopisobaric_\alpha  \left(\dcirc_{RS}^R\right)\cdot 1$ for $R\subsetneq J_0$, this sum
\begin{align*}
& = \sum_{R\subsetneq J_0} 
\left(\left[-e^{-\alpha}\Demopisobaric_\alpha  \left(\dcirc_{RS}^{R}\right)\cdot 1\right]
(1-e^{-\alpha })  r_\alpha (\tau^\circ_{R}|_{J_0}) 
+ 
\dcirc_{RS}^R e^{-\alpha} r_\alpha (\tau^\circ_R|_{J_0}) \right)+ \dcirc_{RS}^{J_0} e^{-\alpha } r_\alpha ( \tau^\circ_{J_0}|_{J_0})\\
 &  =  e^{-\alpha }\left(\sum_{R\subsetneq J_0} \left[-(1-e^{-\alpha }) \Demopisobaric_\alpha  \left(\dcirc_{RS}^{R} 
 \right)\cdot 1+  \dcirc_{RS}^R  \right] r_\alpha (\tau^\circ_R|_{J_0})  +\dcirc_{RS}^{J_0} r_\alpha ( \tau^\circ_{J_0}|_{J_0})\right)\\
 &= e^{-\alpha }\left( \sum_{R\subsetneq J_0} \left[-(1-e^{-\alpha }) \frac{\left(\dcirc_{RS}^R-e^{-\alpha }r_\alpha (\dcirc_{RS}^R)\right)}{1-e^{-\alpha_\alpha }}
 +  \dcirc_{RS}^R 
   \right]r_\alpha (\tau^\circ_R|_{J_0}) +\dcirc_{RS}^{J_0} r_\alpha ( \tau^\circ_{J_0}|_{J_0})\right)\\
 &= e^{-\alpha }\left( \sum_{R\subsetneq J_0} \ \left[\phantom{ [-(1-e^{-\alpha })(\dcirc_{RS}^Q-  } 
e^{-\alpha }r_\alpha (\dcirc_{RS}^R) \phantom{  +  +\dcirc_{RS}^R  )  } \right]r_\alpha (\tau^\circ_R|_{J_0}) +\dcirc_{RS}^{J_0} r_\alpha ( \tau^\circ_{J_0}|_{J_0})\right),
 \end{align*} 
 proving the lemma.
\end{proof}
\junk{
{\color{blue}
\begin{align*}
& =  \sum_{Q\subsetneq J, Q=r_mQ_0,
\atop \mbox{\tiny for some }Q_0}  \dcirc_{RS}^{r_mQ_0} (1-e^{-\alpha_m})  r_m(\tau^\circ_{Q_0}|_{J_0}) + \sum_{|Q|\leq k-1, Q\neq r_mQ_0, \atop \mbox{\tiny for any }Q_0}  \dcirc_{RS}^Q e^{-\alpha_m} r_m(\tau^\circ_Q|_{J_0}) + \dcirc_{RS}^{J_0} e^{-\alpha_m} r_m( \tau^\circ_{J_0}|_{J_0})\\
& = \sum_{|Q|\leq k, Q=r_mQ_0,
\atop \mbox{\tiny for some }Q_0} 
\left[-e^{-\alpha_m}\Demopisobaric_m \left(\dcirc_{RS}^{Q_0}\right)\cdot 1\right]
(1-e^{-\alpha_m})  r_m(\tau^\circ_{Q_0}|_{J_0}) \\
&\phantom{=\sum_{|Q|\leq k, Q=r_mQ_0,
\atop \mbox{\tiny for some }Q_0}  \dcirc_{RS}^{r_mQ_0} (1-e^{-\alpha_m})  r_m(\tau^\circ_{Q_0}|_{J_0}) }
+ \sum_{|Q|\leq k-1, Q\neq r_mQ_0, \atop \mbox{\tiny for any }Q_0}  \dcirc_{RS}^Q e^{-\alpha_m} r_m(\tau^\circ_Q|_{J_0}) + \dcirc_{RS}^{J_0} e^{-\alpha_m} r_m( \tau^\circ_{J_0}|_{J_0})\\
& = \sum_{|Q|\leq k-1} 
\left[-e^{-\alpha_m}\Demopisobaric_m \left(\dcirc_{RS}^{Q}\right)\cdot 1\right]
(1-e^{-\alpha_m})  r_m(\tau^\circ_{Q}|_{J_0}) \\
&\phantom{=\sum_{|Q|\leq k, Q=r_mQ_0,
\atop \mbox{\tiny for some }Q_0}  \dcirc_{RS}^{r_mQ_0} (1-e^{-\alpha_m})  r_m(\tau^\circ_{Q_0}|_{J_0}) }
+ \sum_{|Q|\leq k-1}  \dcirc_{RS}^Q e^{-\alpha_m} r_m(\tau^\circ_Q|_{J_0}) + \dcirc_{RS}^{J_0} e^{-\alpha_m} r_m( \tau^\circ_{J_0}|_{J_0})\\
 &  =  e^{-\alpha_m}\left(\sum_{|Q|\leq k-1} \left[-(1-e^{-\alpha_m}) \Demopisobaric_m \left(\dcirc_{RS}^{Q} 
 \right)\cdot 1+  \dcirc_{RS}^Q   \right] r_m(\tau^\circ_Q|_{J_0})  +\dcirc_{RS}^{J_0} r_m( \tau^\circ_{J_0}|_{J_0})\right)\\
 &= e^{-\alpha_m}\left( \sum_{|Q|\leq k-1} \left[-(1-e^{-\alpha_m}) \frac{\left(\dcirc_{RS}^Q-e^{-\alpha_m}r_m(\dcirc_{RS}^Q)\right)}{1-e^{-\alpha_m}}
 +  \dcirc_{RS}^Q 
   \right]r_m(\tau^\circ_Q|_{J_0}) +\dcirc_{RS}^{J_0} r_m( \tau^\circ_{J_0}|_{J_0})\right)\\
 &= e^{-\alpha_m}\left( \sum_{|Q|\leq k-1} \ \left[\phantom{ [-(1-e^{-\alpha_m})(\dcirc_{RS}^Q-  } 
e^{-\alpha_m}r_m(\dcirc_{RS}^Q) \phantom{  +  +\dcirc_{RS}^Q  )  } \right]r_m(\tau^\circ_Q|_{J_0}) +\dcirc_{RS}^{J_0} r_m( \tau^\circ_{J_0}|_{J_0})\right),
 \end{align*} 
 }

}

We plug the first sum in Lemma \ref{lemma:simplification}
 into~\eqref{eq:inductivestep}, and use the
restrictions found in Table~\ref{table:listrestrictions} and
Lemma~\ref{lemma:simplification}. For the $\{\tau_Q\}$ structure
constants,
\begin{align*}
d_{RS}^J &= \frac{\tau_S|_J \tau_R|_J - \sum_{Q: \ Q\subsetneq J} d_{RS}^Q \tau_Q|_J}{\tau_J|_J}\\
& = \frac{r_\alpha  (\tau_S|_{J_0}\tau_R|_{J_0}) - r_\alpha  \left(\sum_{R\subsetneq J_0} d_{RS}^R \tau_R|_{J_0}\right) - d_{RS}^{J_0} r_\alpha \left(\tau_{J_0}|_{J_0}\right)}{(1-e^{-\alpha })r_\alpha  \left(\tau_{J_0}|_{J_0}\right)}\\
&= \frac{1}{1-e^{-\alpha }} r_\alpha \left(\frac{\tau_S|_{J_0}\tau_R|_{J_0} -\sum_{R\subsetneq J_0} d_{RS}^R \tau_R|_{J_0}}{\tau_{J_0}|_{J_0}}\right) - 
\frac{d_{RS}^{J_0}}{1-e^{-\alpha }}\\
&=  \frac{1}{1-e^{-\alpha }} r_\alpha  \left(d_{RS}^{J_0}\right) - \frac{d_{RS}^{J_0}}{1-e^{-\alpha }}\\
& = \Demop_\alpha  r_\alpha  \left(d_{RS}^{J_0}\right) \quad \mbox{(noting that $-\Demop_\alpha  =\Demop_\alpha  r_\alpha $),}
\end{align*}
as desired. 

Similarly for the $\{\tau_Q^\circ\}$ structure constants, we define $J_0$ such that $J = r_\alpha J_0$ and obtain:
\begin{align*}
\dcirc_{RS}^J &= \frac{\tau^\circ_S|_J \tau^\circ_R|_J - \sum_{Q: \ Q\subsetneq J} \dcirc_{RS}^Q \tau^\circ_Q|_J}{\tau^\circ_J|_J}\qquad\qquad\text{from (\ref{eq:Kinductivestep})} \\
&=  \frac{e^{-2\alpha } r_\alpha (\tau^\circ_{S}|_{J_0}\tau^\circ_{R}|_{J_0})-e^{-\alpha }\left( \sum_{Q\subsetneq J_0}  e^{-\alpha } r_\alpha (\dcirc_{RS}^Q \tau^\circ_Q|_{J_0}) + \dcirc_{RS}^{J_0} r_\alpha ( \tau^\circ_{J_0}|_{J_0})\right) }{ (1-e^{-\alpha }) r_\alpha  (\tau^\circ_{J_0}|_{J_0})}\\
& = e^{-\alpha } \cdot
\frac{e^{-\alpha } r_\alpha (\tau^\circ_{S}|_{J_0}\tau^\circ_{R}|_{J_0}) 
-  \sum_{Q\subsetneq J_0} e^{-\alpha }  r_\alpha (\dcirc_{RS}^Q \tau^\circ_Q|_{J_0}) 
- \dcirc_{RS}^{J_0} r_\alpha ( \tau^\circ_{J_0}|_{J_0}) }{ (1-e^{-\alpha }) r_\alpha  (\tau^\circ_{J_0}|_{J_0})}\\
&=e^{-\alpha } \cdot
\frac{e^{-\alpha } r_\alpha \left(\tau^\circ_{S}|_{J_0}\tau^\circ_{R}|_{J_0}
-  \sum_{Q\subsetneq J_0}  \dcirc_{RS}^Q \tau^\circ_Q|_{J_0}\right) 
- \dcirc_{RS}^{J_0} r_\alpha ( \tau^\circ_{J_0}|_{J_0}) }{ (1-e^{-\alpha }) r_\alpha  (\tau^\circ_{J_0}|_{J_0})}\\
& = e^{-\alpha } \cdot
\frac{e^{-\alpha } r_\alpha  (\dcirc_{RS}^{J_0}\tau^\circ_{J_0}|_{J_0}) 
- \dcirc_{RS}^{J_0} r_\alpha ( \tau^\circ_{J_0}|_{J_0}) }{ (1-e^{-\alpha }) r_\alpha  (\tau^\circ_{J_0}|_{J_0})}\\
&=  -e^{-\alpha } \Demopisobaric_\alpha  (\dcirc_{RS}^{J_0})
 \end{align*}
\junk{{\color{blue}
\begin{align*}
\dcirc_{RS}^J &= \frac{\tau^\circ_S|_J \tau^\circ_R|_J - \sum_{Q: \ |Q|\leq k, R,S\subset Q} \dcirc_{RS}^Q \tau^\circ_Q|_J}{\tau^\circ_J|_J}\\
&=  \frac{e^{-2\alpha_m} r_m(\tau^\circ_{S}|_{J_0}\tau^\circ_{R}|_{J_0})-e^{-\alpha_m}\left( \sum_{|Q|\leq k-1}  e^{-\alpha_m} r_m(\dcirc_{RS}^Q \tau^\circ_Q|_{J_0}) + \dcirc_{RS}^{J_0} r_m( \tau^\circ_{J_0}|_{J_0})\right) }{ (1-e^{-\alpha_m}) r_m (\tau^\circ_{J_0}|_{J_0})}\\
& = e^{-\alpha_m} \cdot
\frac{e^{-\alpha_m} r_m(\tau^\circ_{S}|_{J_0}\tau^\circ_{R}|_{J_0}) 
-  \sum_{|Q|\leq k-1} e^{-\alpha_m}  r_m(\dcirc_{RS}^Q \tau^\circ_Q|_{J_0}) 
- \dcirc_{RS}^{J_0} r_m( \tau^\circ_{J_0}|_{J_0}) }{ (1-e^{-\alpha_m}) r_m (\tau^\circ_{J_0}|_{J_0})}\\
&=e^{-\alpha_m} \cdot
\frac{e^{-\alpha_m} r_m\left(\tau^\circ_{S}|_{J_0}\tau^\circ_{R}|_{J_0}
-  \sum_{|Q|\leq k-1}  \dcirc_{RS}^Q \tau^\circ_Q|_{J_0}\right) 
- \dcirc_{RS}^{J_0} r_m( \tau^\circ_{J_0}|_{J_0}) }{ (1-e^{-\alpha_m}) r_m (\tau^\circ_{J_0}|_{J_0})}\\
& = e^{-\alpha_m} \cdot
\frac{e^{-\alpha_m} r_m (\dcirc_{RS}^{J_0}\tau_{J_0}|_{J_0}) 
- \dcirc_{RS}^{J_0} r_m( \tau^\circ_{J_0}|_{J_0}) }{ (1-e^{-\alpha_m}) r_m (\tau^\circ_{J_0}|_{J_0})}\\
&=  -e^{-\alpha_m} \Demopisobaric_m (\dcirc_{RS}^{J_0})
 \end{align*}
}}

\noindent {\bf \underline{Case 3:} $R, S, J$ all begin with the same letter.}

On the other hand, if $R$ and $S$ both begin with $r_\alpha $, let $R=r_\alpha R_0$ and $S=r_\alpha S_0$.

\begin{align*}
\sum_{Q:\ Q\subsetneq J} d_{RS}^Q \tau_Q|_J &= \sum_{Q:\ Q\subsetneq J} d_{RS}^Q (1-e^{-\alpha })r_\alpha (\tau_{Q_0}|_{J_0}) \\
&= \sum_{Q:\ Q\subsetneq J, \atop Q=r_\alpha Q_0} \left[(1-e^{-\alpha })r_\alpha  \left(d_{R_0S_0}^{Q_0}\right)\right]  \left[(1-e^{-\alpha })r_\alpha \left(\tau_{Q_0}|_{J_0}\right)\right],  \mbox{ by induction}\\
& = (1-e^{-\alpha })^2 r_\alpha \left(\sum_{Q:\ Q=r_\alpha Q_0, \atop Q_0\subsetneq J_0}  d_{R_0S_0}^{Q_0}\tau_{Q_0}|_{J_0}\right)\\
& = (1-e^{-\alpha })^2 r_\alpha \left(\sum_{R:\  R\subsetneq J_0}   d_{R_0S_0}^{R} \tau_{R}|_{J_0}\right).
\end{align*}
Similarly, since $r_\alpha $ is the first letter of both $R$ and $S$, $r_\alpha \in R\cup S \subset Q$. Thus, for each $Q$ containing $R$ and $S$, $Q = r_\alpha Q_0$ for some $Q_0$. It follows that
\begin{align*}
\sum_{Q:\ Q\subsetneq J} \dcirc_{RS}^Q \tau^\circ_Q|_J &=\sum_{Q:\ Q\subsetneq J} \dcirc_{RS}^Q (1-e^{-\alpha })r_\alpha (\tau^\circ_{Q_0}|_{J_0}) \\
& = \sum_{Q:\ Q\subset J} \left[(1-e^{-\alpha })r_\alpha (\dcirc_{R_0S_0}^{Q_0}) \right]
(1-e^{-\alpha })r_\alpha (\tau^\circ_{Q_0}|_{J_0}), \mbox{ by induction} \\
& = (1-e^{-\alpha })^2 r_\alpha  \left(\sum_{R:\  R\subsetneq J_0}   \dcirc_{R_0S_0}^{R} \tau^\circ_{R}|_{J_0}\right).
\end{align*}
We plug these expressions from Lemma~\ref{lemma:simplification} into
Equations~\eqref{eq:inductivestep} and \eqref{eq:Kinductivestep}
when $r_\alpha $ is the first letter of all three words $J, U,$ and $R$, and use the restrictions in
Table~\ref{table:listrestrictions} to obtain
\begin{align*}
d_{RS}^J &= \frac{(1-e^{-\alpha })^2 r_\alpha (\tau_{S_0}|_{J_0}\tau_{R_0}|_{J_0})-(1-e^{-\alpha })^2 r_\alpha \left(\sum_{R\subsetneq J_0}  d_{R_0S_0}^{R} \tau_R|_{J_0}\right)}{ (1-e^{-\alpha }) r_\alpha  \tau_{J_0}|_{J_0}}\\
&= (1-e^{-\alpha }) r_\alpha  \left(\frac{\tau_{S_0}|_{J_0}\tau_{R_0}|_{J_0}-\left(\sum_{R\subsetneq J_0  } d_{R_0S_0}^{R} \tau_R|_{J_0}\right)}{\tau_{J_0}|_{J_0}}\right)\\
&= (1-e^{-\alpha })r_\alpha  \left(d_{R_0S_0}^{J_0}\right)
\end{align*}
where the last step follows again by the inductive hypothesis. Similarly, using the restrictions table for $\tau^\circ_S$ and $\tau^\circ_R$ restricted to $J$, 
 \begin{align*}
 \dcirc_{RS}^J &=  \frac{((1-e^{-\alpha })^2 r_\alpha (\tau^\circ_{S_0}|_{J_0}\tau^\circ_{R_0}|_{J_0})-(1-e^{-\alpha })^2 r_\alpha \left(\sum_{R\subsetneq J_0}  \dcirc_{R_0S_0}^{R} \tau_R|_{J_0}\right)}{ (1-e^{-\alpha }) r_\alpha  (\tau^\circ_{J_0}|_{J_0})}\\
 &=(1-e^{-\alpha })r_\alpha  \left( \frac{\tau^\circ_{S_0}|_{J_0}\tau^\circ_{R_0}|_{J_0}-\sum_{R\subsetneq J_0}  \dcirc_{R_0S_0}^{R} \tau_R|_{J_0}}{\tau^\circ_{J_0}|_{J_0}}\right)\\
 &= (1-e^{-\alpha })r_\alpha  \left(\dcirc_{R_0S_0}^{J_0}\right).
 \end{align*}


\noindent {\bf \underline{Case 4:} Exactly one of $R$ or $S$ begin with the same letter as $J$.}

Finally, we consider the case that $R=r_\alpha R_0$ for some $R_0$, while $S$ does not begin with $r_\alpha $ (or, symmetrically, if $S$ begins with $r_\alpha $ but $R$ does not). Recall that $\tau_Q|_J\neq 0$ implies $Q\subset J$, so $R\subset Q$ implies $Q=r_\alpha Q_0$ for some $Q_0$. Thus by the inductive assumption and Table~\ref{table:listrestrictions},
\begin{align*}
\sum_{Q:\ Q\subsetneq J}
d_{RS}^Q \tau_Q|_J &=  \sum_{Q:\ Q\subsetneq J\atop Q=r_\alpha Q_0} 
\left[r_\alpha  \left(d_{R_0S}^{Q_0}\right) \right]
\left[(1-e^{-\alpha })r_\alpha \left(\tau_{Q_0}|_{J_0}\right) \right],\quad\mbox{and}
\\
\sum_{Q:\ Q\subsetneq J}
 \dcirc_{RS}^Q \tau^\circ_Q|_J &=  \sum_{Q:\ Q\subsetneq J\atop Q=r_\alpha Q_0} 
 \left[e^{-\alpha}\, r_\alpha  \left(\dcirc_{R_0S}^{Q_0}\right)\right] \left[(1-e^{-\alpha })r_\alpha (\tau^\circ_{Q_0}|_{J_0})\right]
\end{align*}

Using these equalities together with the restrictions in Equations \eqref{eq:inductivestep} and \eqref{eq:Kinductivestep},
\begin{align*}
d_{RS}^J &= \frac{(1-e^{-\alpha }) r_\alpha (\tau_{S_0}|_{J_0}\tau_{R_0}|_{J_0})-(1-e^{-\alpha })r_\alpha \left(\sum_{R\subsetneq J_0}  d_{R_0S_0}^R \tau_R|_{J_0}\right)}{ (1-e^{-\alpha }) r_\alpha  \tau_{J_0}|_{J_0}}\\
&=r_\alpha  \left(\frac{\tau_{S_0}|_{J_0}\tau_{R_0}|_{J_0}-\left(\sum_{R\subsetneq J_0} d_{R_0S_0}^R\tau_R|_{J_0}\right)}{\tau_{J_0}|_{J_0}}\right)\\
&= r_\alpha  \left(d_{R_0S_0}^{J_0}\right).
\end{align*}
and
\begin{align*}
\dcirc_{RS}^J &=  \frac{e^{-\alpha } (1-e^{-\alpha }) r_\alpha (\tau^\circ_{S_0}|_{J_0}\tau^\circ_{R_0}|_{J_0})-e^{-\alpha } (1-e^{-\alpha })r_\alpha \left(\sum_{R\subsetneq J_0}  \dcirc_{R_0S_0}^R \tau^\circ_R|_{J_0}\right)}{  (1-e^{-\alpha })r_\alpha (\tau^\circ_{J_0}|_{J_0})}\\
&=e^{-\alpha } r_\alpha  \left(\frac{\tau^\circ_{S_0}|_{J_0}\tau^\circ_{R_0}|_{J_0}-\left(\sum_{R\subsetneq J_0} \dcirc_{R_0S_0}^R\tau^\circ_R|_{J_0}\right)}{\tau^\circ_{J_0}|_{J_0}}\right)\\
&= e^{-\alpha} r_\alpha  \left( \dcirc_{R_0S_0}^{J_0}\right).
\end{align*}
This completes the proof in all cases.
\end{proof}


\end{document}